\documentclass[a4paper]{article} 
\usepackage[left=1.5cm,right=1.5cm,top=2.9cm]{geometry}
\usepackage{authblk}
\usepackage{CJKutf8}
\usepackage{yfonts}
\usepackage{lipsum}
\usepackage[utf8]{inputenc}

\usepackage{mathtools}

\title{Splitting algorithm and normed convergence for drawing\\ the random Loewner curves}
\author{Jiaming Chen\thanks{chen.jiaming@cims.nyu.edu}\quad and\quad Vlad Margarint\thanks{vmargari@charlotte.edu}  }
\affil{$\prescript{*}{}{\text{Courant}}$ Institute of Mathematical Sciences, New York University\par $\prescript{\dagger}{}{\text{Department}}$ of Mathematics, University of North Carolina Charlotte }
\date{\today}

\usepackage{graphicx}
\usepackage{amssymb}
\usepackage{amsmath}
\usepackage{amsthm}
\usepackage{tikz}
\usepackage{todonotes}
\usepackage{enumitem}
\usepackage{verbatim}

\usepackage{centernot}
\usepackage{mathtools}
\usepackage{ stmaryrd }

\usepackage{hyperref}
\hypersetup{allcolors=black,
pdftitle={SLE theory and splitting algorithm},
pdfpagemode=FullScreen}
\numberwithin{equation}{section}
\usepackage{bbm}

\usepackage{braket}
\usepackage{amsmath,amsthm,amssymb}
\usepackage{physics}
\usepackage{mathtools}
\usepackage{bm}
\usepackage{esvect}
\usepackage[english]{babel}
\usepackage{extarrows} 
\usepackage{mathrsfs}
\usepackage{esint}

\usepackage{graphicx}  
\usepackage{subcaption}

\usepackage{setspace}
\usepackage{tocloft}
\usepackage{titlesec}
\titleformat{\subsection}[runin]
  {\normalfont\large\bfseries}{\thesubsection}{1em}{}

\usepackage{amsmath}

\numberwithin{equation}{section}
\usepackage{bbm}

\usepackage{braket}
\usepackage{amsmath,amsthm,amssymb}
\usepackage{physics}
\usepackage{mathtools}
\usepackage{bm}
\usepackage{esvect}
\usepackage[english]{babel}

\newtheorem{theorem}{Theorem}[section]

\newtheorem{proposition}[theorem]{Proposition}

\theoremstyle{definition}
\newtheorem{definition}[theorem]{Definition}

\theoremstyle{remark}

\usepackage[english]{babel}

\newcommand\normx[1]{\lVert#1\rVert}
\newcommand\normy[1]{\big\lVert#1\big\rVert}

\makeatletter
\renewenvironment{proof}[1][\proofname]{%
  \par\pushQED{\qed}\normalfont%
  \topsep6\p@\@plus6\p@\relax
  \trivlist\item[\hskip\labelsep\bfseries#1\@addpunct{.}]%
  \ignorespaces
}{%
  \popQED\endtrivlist\@endpefalse
}
\makeatother

\begin{document}
\maketitle

\begin{abstract}
    Recent advances in Schramm–Loewner evolution have driven increasing interest in non-standard Loewner flows. In this work, we propose a novel splitting algorithm to simulate random Loewner curves with rigorous convergence analysis in sup-norm and $L^p$. The algorithm is further extended to explore fractional SLE, driven by fractional Brownian motion, and noise-reinforced SLE, incorporating the effect on long-term memory. These exploratory and numerical extensions enable theoretical predictions on fractal dimensions and other statistical phenomena, providing new insights into such dynamics and opening directions for future research.

\end{abstract}

    \tableofcontents


\section{Introduction}
    Motivated by the study of Charles Loewner \cite{Loewner} for the Bieberbach conjecture \cite{de Branges}, the Schramm--Loewner evolution (SLE) \cite{Schramm} was created to describe the scaling interface of planar models in statistical physics, organizing the traces into universality classes distinguished by one single parameter \cite{Cardy2}. The loop-erased random walk \cite{Lawler/Viklund}, e.g.~corresponds to SLE(2), the critical Ising model \cite{Chelkak/Duminil-Copin/Hongler/Kemppainen/Smirnov} to SLE(3), the Gaussian free field \cite{Schramm/Sheffield} to SLE(4) and uniform spanning tree \cite{Liu/Wu} to SLE(8). This varying parameter is classified by the strength of diffusivity of the Brownian motion that generates the SLE and is denoted by $\kappa$. Pointing to different lattice models, the behavior of SLE($\kappa$) significantly changes when $\kappa$ is located in different domains. Using a Bessel-type approach \cite{Rohde/Schramm}, one can show that SLE$(\kappa)$ is simple when $0<\kappa\leq4$, swallowing \cite{Lawler} when $4<\kappa<8$, transient and space-filling \cite{Lind/Rohde} when $\kappa\geq8$. Such drastically different behavior of SLE with different $\kappa$ is also seen from other observables: Left passage probability \cite{Najafi}, the Fokker--Plank equation \cite{Najafi2} and fractal dimension of the Loewner traces \cite{Beffara}.\par
    The development of SLE has thus made it desirable to simulate such random curves for better understanding their behaviors. As a first order stochastic ODE, the Loewner equation is naturally approximated by Euler's method, see illustrations e.g.~by Vincent Beffara \cite{Beffara2}. Such approximation is admittedly good but it does not show the generated Loewner curve rather than a neighborhood of it. A second method for drawing SLE was suggested by Marshall/Rohde \cite{Marshall/Rohde} via the classic zipper algorithm \cite{Kuhnau}. Given discretization of the driving force, the zipper algorithm builds conformal map that takes those discrete points to the trace of a Jordan curve, and whence simulates the Loewner flow \cite{Kennedy0}. Its Carathéodory convergence \cite{Lawler} and Cauchy transform of marginal measures \cite{Bauer} are known only as convergence of Loewner flows rather than curves. Viewed as compact sets in the upper half plane, the outputs of zipper algorithm converge to SLE under Hausdorff metric in the $\kappa\leq4$ simple phase. The sup-norm convergence of zipper algorithm was not known until \cite{Tran}. Nonetheless, although well approximating SLE, the curves drawn from this algorithm are piecewise smooth \cite{Kennedy00} and simple, leading inevitably to skepticism \cite[p. 602]{Tran} at their first appearance.\par
    In this work, we propose and study the splitting algorithm via the Ninomiya--Victoir scheme for drawing SLE. Inspired by rough path theory \cite{Friz/Victoir}, which views SDEs as functions that map Brownian motion to continuous fractal paths, the splitting algorithm decomposes the Loewner flow into iterated stochastic integrals of the driving force \cite{Foster/Reis/Strange}, obtaining relatively good approximation, see e.g.~Figure \ref{fig: box-counting and yardstick} or \cite{Foster}. We rigorously prove the pathwise sup-norm and $L^p$ convergences of the splitting output to SLE, guaranteeing the effectiveness of our algorithm. On the other hand, due to the iterated integral nature of the splitting algorithm, the exact rate of convergence should relate to the Lévy area \cite{Clark/Cameron} which is difficult and is beyond the scope of this paper. Furthermore, the splitting algorithm has also been manifested in various models such as the stochastic Landau--Lifshitz equation \cite{Ableidinger/Buckwar}, diffusion processes with Lipschitz drift \cite{Buckwar/Samson/Tamborrino/Tubikanec} and the stochastic oscillatory dynamics \cite{Buckwar/Tamborrino/Tubikanec} with hypoelliptic steps \cite{Petersen} and is geometrically ergodic \cite{Mattingly/Stuart/Higham}.\par
    To this recognition, we apply the splitting algorithm to non-standard Loewner evolutions which are much more general than SLE, and make observations on the statistical phenomena that are still open and very difficult to solve. Such notion of non-standard Loewner evolution has appeared on several fronts: Smooth criteria for traces driven by finite variation process are given in \cite{Shekhar/Tran/Wang}; Existence of continuous fractal curves is later extended to semimartingale driving force \cite{Margarint/Shekhar/Yuan} and further to complex-valued driving force \cite{Lind/Utley} with the aid of Itô's formula; When the Loewner flow is driven by Lévy processes, càdlàg traces were discussed in \cite{Peltola/Schreuder} with careful derivative estimates. Replacing standard Brownian motion by variant driving forces yields a very complicated multifractal structure: The traces can become highly fractal and they can be either tree-like, forest-like or looptree-like. The conformal invariance of SLE, which was vital to Brownian intersection exponents \cite{Duplantier/Kwon} and conformal field theory \cite{Ang/Remy/Sun}, is also destroyed in general. It is not at all clear how to rigorously proceed with their theoretical aspects - One has to to invent new tools such as the supermartingale technique in \cite[Lemma 5.6.8]{Cohen/Elliott}, \cite[Proposition 3.13]{Peltola/Schreuder}. It is then illuminating to use high order numerical methods, e.g.~the splitting algorithm, for accurate theoretical predictions.\par
    Beyond the above cádlág semimartingale driving functions, some recent advances on non-semimartingale driven SLE have made striking observations on both numerical and theoretical aspects. Motivated and borrowed techniques from SLE driven by Lévy flights \cite{Oikonomou/Rushkin/Gruzberg/Kadanoff}, the subordinated SLE was developed in \cite{Ghasemi Nezhadhaghighi/Rajabpour/Rouhani} by considering time as a monotonic stochastic parameter. The Loewner evolution driven by a superposition of Brownian motion and Lévy pure jump was mentioned in \cite{Rushkin/Oikonomou/Kadanoff/Gruzberg}, where the branching phenomenon and phase transition were first observed in \cite{Oikonomou/Rushkin/Gruzberg/Kadanoff} and later rigorously shown in \cite{Guan/Winkel}. Anisotropic SLE was investigated in \cite{Credidio/Moreira/Herrmann/Andrade} where the driving force obeys power-law in time, leading to an explicit characterization of fractional Brownian driving force \cite[Section II]{Tizdast/Ebadi/Cheraghalizadeh/Najafi/Andrade/Herrmann}. It is there significant that aptly modified deterministic drift in Loewner SDE yields scale-invariant Loewner flow, from where the generated Loewner curve has been numerically observed self-similar \cite[Section IV]{Tizdast/Ebadi/Cheraghalizadeh/Najafi/Andrade/Herrmann}, leading to distinguishable contrasts to \cite{Margarint/Shekhar/Yuan,Peltola/Schreuder,Shekhar/Tran/Wang} and pointing to theoretical new horizon with optimistic geometric predictions. In light of the still-existing computational uncertainties, we revisit this fractional SLE model in this work, and we use our splitting algorithm for drawing its random traces and analyzing its fractal dimension. Here we adopt the box-counting method instead of the yardstick method for computing the dimension of fractal curves, see Figure \ref{fig: box-counting and yardstick} for illustration. Not quite surprisingly, our drawings of fractional SLE, see Figure \ref{fig: Fractional SLE}, via splitting algorithm are almost the same as those in \cite{Tizdast/Ebadi/Cheraghalizadeh/Najafi/Andrade/Herrmann}: The curves maintain self-similar shape, get smoother at higher Hurst exponents, and more fractal as $\kappa$ increases. Furthermore, the numerical results of computing the fractal dimension also shows the same monotonicity trend as the $\kappa$ and Hurst exponent vary, further confirming the optimism of the conclusions in \cite[Section V]{Tizdast/Ebadi/Cheraghalizadeh/Najafi/Andrade/Herrmann}.\par
    Apart from the model of fractional SLE, we look ahead to formulate an interplay between non-standard SLE and stochastic reinforcement. This fascinating mechanism allows processes to have long-term memory and that their future dynamics is strongly affected by their past trajectories. The edge-reinforced random walk \cite{Merkl/Rolles}, for instance, describes a particle traversing on a graph where it is more likely to repeat the edges which has been visited before. Whether this particle is attracted to a single edge depends on the exact reinforcement strength \cite{Limic/Tarres}. Recent results have pointed out that such reinforced random walks can be decomposed into a mixture of the continuous-time vertex-reinforced jump process \cite{Sabot/Tarres} where the particle jumps according to an exponential law with reinforced weights in time. Novel variant models have also been proposed such as the trace-reinforced ants process \cite{Kious/Mailler/Schapira} where ants embark on a journey to search food from their nest. It has been shown that these ants will not find the shortest path \cite{Kious/Mailler/Schapira2} with some still left-open questions. See also \cite{Pemantle} for a great survey on various other reinforced models with self-interaction. Such reinforcement provides a new angle where stochastic processes undergo distortions by their memories of the past. Such memory need not be only on their past trajectories in state spaces, but could also be less tangible such as the while noise in the past \cite{Bertoin0,Bertoin2}.\par
    In this direction, Jean Bertoin \cite{Bertoin} innovatively designed the so called noise-reinforced Brownian motion whose fluctuation is repeated in time and corresponds to the scaling limit of some step-reinforced random walks \cite{Bertenghi/Rosales-Ortiz}. This type of memory introduces a novel variant of the standard Brownian motion. And there has been no discussion, to the best of our knowledge, on the non-standard SLE driven by such reinforced Brownian motion with memory. This is a natural theater where the splitting algorithm comes into play. The first question to ask is whether the induced Loewner evolution inherits, in some sense, such long-term memory from the noise-reinforced Brownian motion. Via applying the splitting algorithm, we observe in Figure \ref{fig: noise-reinforced SLE drawings} that the traces of such noise-reinforced SLE do have the tendency to repeat their previous twists compared to the standard SLE. The nontrivial drift term in the Itô semimartingale decomposition of the noise-reinforced Brownian motion breaks down the classic martingale approach, making rigorous arguments hard to proceed. Nevertheless, SLE based methods such as the splitting algorithm are suitable candidates for drawing the noise-reinforced SLE which is again semimartingale with respect to Brownian motion. And with the aid of effective numerical methods e.g.~the splitting algorithm, we will know what exactly to expect from the numerical observations of the geometry of such noise-reinforced SLE.\par
    This paper is organized as follows. In Section \ref{sec: splitting algorithm} we introduce the mechanism of the splitting algorithm and how it applies to the Loewner SDE, following a brief review of the general SLE theory. We then discuss how such splitting algorithm can be extended to some non-standard Loewner evolutions such as the noise-reinforced SLE and fractional SLE. In Section \ref{sec: convergence of splitting algorithm to SLE} we demonstrate that such splitting algorithm is effective in the context of SLE by showing its sup-norm and $L^p$ convergence to the random Loewner curves. In practical simulations, we discretize the driving forces and approximate them via polynomial interpolation. This auxiliary tool is combined with the splitting algorithm to ease computation. And we thus include the convergence analysis of arbitrary $p$th order polynomial interpolation at the end of this section in the context of standard SLE, where we also incorporate exploratory and numerical discussions for such non-standard SLE. In Section \ref{sec: Conclusion remarks}, the drawings and numerical observations are presented in various figures. Significant statistical and geometric phenomena are investigated with discussions on the fractal dimensions, shedding light on the theory of stochastic reinforcement and fractional stochastic processes. Future perspectives on the quantitative and exact rate of algorithm convergence is also mentioned, pointing to new research directions, concluding this paper.

\section{Ninomiya--Victoir splitting algorithm to Schramm--Loewner evolution}\label{sec: splitting algorithm}
    In the section, we review briefly the well-studied SLE and introduce an SDE-based stochastic algorithm to simulate the geometry of this random evolution. Besides verifying its convergence, we also point out that this new algorithm could be applied to non-standard SLE, distinguishing itself from other frequently used numerical schemes.

\subsection{Random Loewner curves from Schramm--Loewner evolution.}
    Let $\mathbb{H}\coloneqq\{z\in\mathbb{C}:\,\Im~z>0\}$ be the upper half-plane and $\overline{\mathbb{H}}$ its closure in $\mathbb{C}$. We say a set $K\subseteq\overline{\mathbb{H}}$ is a hull if $K$ is compact and $\mathbb{H}\backslash K$ is simply connected in $\mathbb{C}$. For any such hull $K$, we can find a unique conformal onto map $g_K:\mathbb{H}\backslash K\to\mathbb{H}$ such that
    \begin{equation*}
        g_K(z) = z + a_1\frac{1}{z} + a_2\frac{1}{z^2} + \cdots \qquad\text{near}\quad\infty,   
    \end{equation*}
    yielding the hydrodynamic limit $z-g_K(z)\to0$ as $\abs{z}\to\infty$. One should remark \cite[Lemma 4.2]{Kemppainen} that $a_1(K)$ is real non-negative and vanishes only if $g_K=Id$ on $\mathbb{H}$.\par
    When there is a growing family $(K_t)_{t\in\mathbb{R}_+}$ of hulls, we say this family has continuous local growth if the radius of $g_t(K_{t+s}\backslash K_t)$ tends to $0$ as $s\to0+$ uniformly in $t\in\mathbb{R}_+$. In this scenario, we conveniently reparametrize $(K_t)_{t\in\mathbb{R}_+}$ so that $a_1(K_t)=2t$ for all $t$. And it is thus convenient to write $g_t\coloneqq g_{K_t}$ on $\mathbb{H}\backslash K_t$. When $(K_t)_{t\in\mathbb{R}_+}$ grows continuously, there exists a real-valued continuous $(\lambda_t)_{t\in\mathbb{R}_+}$ satisfying the forward Loewner equation
    \begin{equation}\label{eqn: g_t, forward Loewner equation}
        \partial_t g_t(z) = \frac{2}{g_t(z)-\lambda_t},\qquad\forall~z\in\mathbb{H}\backslash K_t,
    \end{equation}
    with $g_0=Id$ on $\mathbb{H}$. In fact, there is a one-to-one correspondence \cite[Theorem 4.2]{Kemppainen} between the local growth and the continuous $(\lambda_t)_{t\in\mathbb{R}_+}$. To this recognition, the process $(\lambda_t)_{t\in\mathbb{R}_+}$ essentially determines the evolution $(g_t)_{t\in\mathbb{R}_+}$ as well as the hulls $(K_t)_{t\in\mathbb{R}_+}$. And hence we call $(\lambda_t)_{t\in\mathbb{R}_+}$ the driving force of $(g_t)_{t\in\mathbb{R}_+}$.\par
    Motivated by the study of theoretical physics, SLE was created to describe scaling limit of interface models, where one takes into account the randomness of $(\lambda_t)_{t\in\mathbb{R}_+}$. And the description of such behavior requires a probability space $(\Omega,\mathbb{P})$ with filtration $(\mathscr{F}_t)_{t\in\mathbb{R}_+}$ large enough to support a standard Brownian motion $(B_t)_{t\in\mathbb{R}_+}$. Viewed as a stochastic process on $(\Omega,\mathbb{P})$, $(\lambda_t)_{t\in\mathbb{R}_+}$ is assumed to be $(\mathscr{F}_t)_{t\in\mathbb{R}_+}$-adapted and is often called the driving process.\par
    For certain driving processes $(\lambda_t)_{t\in\mathbb{R}_+}$, the hulls $(K_t)_{t\in\mathbb{R}_+}$ are generated by a curve $\gamma:\mathbb{R}_+\to\overline{\mathbb{H}}$, i.e.~$\mathbb{H}\backslash K_t$ is the unbounded component of $\mathbb{H}\backslash\gamma[0,t]$ for all $t\in\mathbb{R}_+$. And for all such deterministic or random evolutions $(g_t)_{t\in\mathbb{R}_+}$, the curve $\gamma$ can be traced by $\gamma(t)=\lim_{\epsilon\searrow0} g^{-1}_t(\lambda_t+i\epsilon)$. And we call this $\gamma$ the random Loewner curve. It turns out that \cite{Schramm} only when $\lambda_t=\sqrt{\kappa}B_t$ for all $t\in\mathbb{R}_+$ and $\kappa\geq0$, the Loewner curve $\gamma$ exists and uniquely satisfies the conformal invariance and domain Markov property.\par
    The regularity of $\gamma$ is nevertheless complicated. First, $\gamma$ is fractal when it is generated by $(\sqrt{\kappa}B_t)_{t\in\mathbb{R}_+}$. Second, under the same driving process this Loewner curve is simple $\mathbb{P}$-a.s. when $0\leq\kappa\leq4$. And $\cup_{t\geq0}K_t=\mathbb{H}$ but $\gamma$ avoids almost all $z\in\mathbb{H}$ when $4<\kappa<8$, see \cite{Lawler/Schramm/Werner,Rohde/Schramm}. In contrast, the Loewner curve fills $\mathbb{H}$ when $\gamma\geq8$, $\mathbb{P}$-a.s.\par
    The three phases of $\gamma$ are then respectively named the simple, the swallowing and the space-filling phase. Depending on the value of $\kappa$, it is hard to draw precisely the trace of the Loewner curve due to its phase transition and fractality. In this paper, we propose and verify that the splitting algorithm is an effective method for such simulation. To this end, one needs to define the reverse Loewner equation
    \begin{equation}\label{eqn: h_t, reverse Loewner equation}
        \partial_t h_t(z) = \frac{-2}{h_t(z) - \sqrt{\kappa}(B_{T-t}-B_T)},\qquad\forall~z\in\mathbb{H},
    \end{equation}
    up to time horizon $T>0$. This Loewner evolution $(h_t)_{t\in[0,T]}$, running in reverse time direction, generates a curve $\eta:[0,T]\to\overline{\mathbb{H}}$. In fact, it is shown \cite{Rohde/Zhan} that $\eta$ corresponds to $\gamma$ in the way that they have the same distribution on $[0,T]$ modulo real scalar shift $\sqrt{\kappa}B_T$. And thus we call $\eta$ the shifted Loewner curve which preserves statistical properties of $\gamma$ and is the main object of our further discussion.

\subsection{Description of the Ninomiya--Victoir splitting algorithm.}
    Before drawing the random Loewner curve $\eta$, we first rephrase the evolutions $(g_t)_{t\in[0,T]}$ and $(h_t)_{t\in[0,T]}$ in convenient fashions. Setting $\Tilde{g}_t(z)=g_t(z)-\sqrt{\kappa}B_t$ for all $z\in\mathbb{H}\backslash K_t$ and $t\in[0,T]$, we can rewrite (\ref{eqn: g_t, forward Loewner equation}) as
    \begin{equation}
        d\Tilde{g}_t(z) = \frac{2}{\Tilde{g}_t(z)}\,dt - \sqrt{\kappa}\,dB_t,\qquad\forall~z\in\mathbb{H}\backslash K_t
    \end{equation}
    with $\Tilde{g}_0=Id$ on $\mathbb{H}$. On the other hand, we consider $Z_t(z)=h_t(z)-\sqrt{\kappa}B_t$ for all $z\in\mathbb{H}$ and $t\in[0,T]$. Then we know by \cite[Eqn. (6.5)]{Foster/Lyons/Margarint}, the reverse equation (\ref{eqn: h_t, reverse Loewner equation}) can be written as
    \begin{equation}\label{eqn: SDE of Z_t}
        dZ_t = -\frac{2}{Z_t}\,dt + \sqrt{\kappa}\,dB_t,\qquad\text{with}\quad Z_0=iy,
    \end{equation}
    for some sufficiently small initial value $y>0$. 

    \begin{definition}{\textbf{Stochastic Splitting Algorithm.}}\label{def: N-V splitting algorithm}
        Consider the complex-valued SDE with the form
        \[
            dW_t = L_0(W_t)\,dt + L_1(W_t)\,dB_t,\qquad W_0=\xi,
        \]
        with $\xi\in\mathbb{C}$ and smooth vector fields $L_1,L_2$ on $\mathbb{C}$. We also use $\exp(tL_i)z_0$ to denote the unique solution to the ODE flow $\Dot{z}_t=tL_i(z_t)$ starting from $z_0\in\mathbb{C}$ at time $t$, for any $i=1,2$. At each $N$th step, we choose a uniform partition $\mathcal{D}_N\coloneqq\{ t_0=0,t_1,\ldots,t_N=T \}$ with mesh size $h_k\coloneqq t_{k+1}-t_k$ for all $k<N$. The size of the mesh $\abs{h_k}$ is free to choose, but formally $\mathcal{D}_N$ gets finer as the step $N$ advances.\par
        We whence numerically approximate the SDE with the discrete points $(\widetilde{W}_{t_k})_{0\leq k\leq N}$ by $\widetilde{W}_{t_0}=\xi$ and 
        \[
            \widetilde{W}_{t_{k+1}} = \exp\big(\tfrac{1}{2}h_kL_0\big)\exp\big(B_{t_k,t_k+h_k}L_1\big)\exp\big(\tfrac{1}{2}h_kL_0\big)\widetilde{W}_{t_k},\qquad\forall~k=0,\ldots,N-1,
        \]
        with $B_{t_k,t_k+h_k}\coloneqq B_{t_{k+1}}-B_{t_k}$ in standard notation \cite[Chapter 2]{LeGall}. In fact, from the discrete $(\widetilde{W}_{t_k})_{0\leq k\leq N}$ the splitting algorithm further yields a continuous approximation $(\widetilde{W}_t)_{t\in[0,T]}$ in the integral form
        \[
            \widetilde{W}_t = \xi + \frac{1}{2}\int_0^t L_0(\widetilde{W}^{(0)}_s) \,ds + \int_0^t L_1(\widetilde{W}^{(1)}_s)\,dB_s +  \frac{1}{2}\int_0^t L_0(\widetilde{W}^{(2)}_s) \,ds,\qquad\forall~t\in[0,T],
        \]
        where on each mesh interval $t_k\leq t<t_{k+1}$ one defines
        \[
            \widetilde{W}^{(0)}_t\coloneqq \exp\big(\tfrac{1}{2}(t-t_k)L_0\big)\widetilde{W}_{t_k},\quad \widetilde{W}^{(1)}_t\coloneqq\exp\big(B_{t_k,t_k+h_k}L_1\big)\widetilde{W}^{(0)}_{t_{k+1}},\quad \widetilde{W}^{(2)}_t\coloneqq \exp\big(\tfrac{1}{2}(t-t_k)L_0\big)\widetilde{W}^{(1)}_{t_{k+1}}.
        \]
    \end{definition}
    In light of the same distribution between $\gamma$ and $\eta$, it suffices to draw the shifted Loewner curve $\eta$ from the SDE (\ref{eqn: SDE of Z_t}) via the splitting algorithm described in Definition \ref{def: N-V splitting algorithm}. Here the Loewner evolution specifies $L_0(z)=-2/z$ and $L_1(z)=\sqrt{\kappa}$ for all $z\in\mathbb{H}$. And to clarify notation, we use $\{t_k(t),t_{k+1}(t)\}\subseteq\mathcal{D}_N$ for the neighboring points of $t$ in $\mathcal{D}_N$, i.e.~$t_k(t)\leq t<t_{k+1}(t)$. We also use $\abs{\mathcal{D}_N}$ for the mesh size of $\mathcal{D}_N$ and $\normx{\vdot}_T$ for the sup-norm on the interval $[0,T]$. To ease notation, we also use $\normx{\sqrt{\kappa}B_{t_k+\vdot}}_{h_k}$ for $\sup_{0\leq t\leq T}\abs{\sqrt{\kappa}(B_t-B_{t_k(t)})}$. \par
    When applying the splitting algorithm to SLE, at each $N$th step we initiate from $\xi=iy_N$ where $y_N=N^{-1/2}$. The algorithm then yields a discrete $(\widetilde{Z}_{t_k})_{0\leq k\leq N}$ and a continuous $(\widetilde{Z}_t)_{t\in[0,T]}$, both of which approximate $\eta[0,T]$ in their respective sense. In Section \ref{sec: convergence of splitting algorithm to SLE} we show the convergence of this algorithm in the context of SLE, and we state the main results here.

    \begin{theorem}\label{thm: sup-norm convergence}
        \normalfont
        Let $\eta[0,T]$ be the shifted Loewner curve for $\kappa\neq8$. If the mesh size $\abs{\mathcal{D}_N}\to0+$ with $N\to\infty$ at a rate $\abs{\mathcal{D}_N}=o(N^{-3})$, then
        \[
            \mathbb{P}\big(\normy{\eta-\widetilde{Z}(iy_N)}_T\leq\varphi_1(N)\big)\geq1-\varphi_2(N),\qquad\forall~N\geq1,
        \]
        for some monotonically decreasing $\varphi_i:\mathbb{N}_+\to\mathbb{R}_+$ with $\varphi_i\to0+$ for each $i=1,2$.
    \end{theorem}

    The above theorem guarantees that the output $(\widetilde{Z}_t)_{t\in[0,T]}$ simulates SLE$(\kappa)$ and is further refined as the step $N$ advances. Apart from this sup-norm convergence local on $[0,T]$, we should the splitting algorithm is effective in an alternative $L^p$ point of view. Namely, we have the following result.

    \begin{theorem}\label{thm: Lp-norm convergence}
        \normalfont
        Let $\eta[0,T]$ be the shifted Loewner curve for $\kappa\neq8$. If the mesh size $\abs{\mathcal{D}_N}\to0+$ with $N\to\infty$ at a rate $\abs{\mathcal{D}_N}=o(N^{-3})$ and $p\geq2$, then
        \[
            \mathbb{E}\bigg[\int_0^T \big|\eta_t-\widetilde{Z}_t (iy_N)\big|^p\,dt\vdot I_{\mathcal{A}_N}\bigg]\leq\psi_1(N),\qquad\text{with}\quad\mathbb{P}(\mathcal{A}_N) \geq 1-\psi_2(N),\qquad\forall~N\geq1,
        \]
        for some monotonically decreasing $\psi_i:\mathbb{N}_+\to\mathbb{R}_+$ with $\psi_i\to0+$ for each $i=1,2$.
    \end{theorem}

    At each $N$th step, we run the algorithm from the initial value $iy_N$ with $y_N=N^{-1/2}$. Here we cannot choose $y_N$ constant over $N$ since otherwise the convergence pattern breaks down. One should also remark that $y_N$ should not decay too fast to destroy certain inequalities necessary for the proof in Section \ref{sec: convergence of splitting algorithm to SLE}.\par
    Before going through the proof of main results, let us diverge to some insights on the extension of this splitting algorithm to non-standard Loewner evolutions, with applications to their statistical mechanics phenomena, such as the fractal dimension of their associated random Loewner curves.

\subsection{Non-standard Loewner chains driven by semimartingales.}\label{sec: Non-standard Loewner chains driven by semimartingales}
    Heuristically, the random vibration of Loewner curves inherits from the randomness of their driving forces. By the wide applicability such as the random dendritic growth at off-critical phase \cite{Johansson/Sola} and the diffusion-limited aggregation \cite{Miller/Sheffield}, one is motivated to consider more general driving forces than the ubiquitous Brownian motion. Lévy process, e.g., is shown \cite{Peltola/Schreuder} to drive locally connected $(\partial\mathbb{H}\backslash K_t)_{t\in\mathbb{R}_+}$. In \cite{Margarint/Shekhar/Yuan}, it was shown that certain continuous semimartingales draw continuous curves.\par
    In particular, the \textit{noise-reinforced Brownian motion} \cite{Bertoin} extends $(B_t)_{t\in\mathbb{R}_+}$ in semimartingale form and relates the notion of stochastic reinforcement. This model allows noise to repeat itself infinitesimally, thus creating strengthened fluctuations as time flows. Readers are also referred to \cite{Bertoin2} for exposition on its Bessel norm process and to \cite{Rosales-Ortiz} on the similarly reinforced Lévy process.\par
    Concerning the mathématique, we start with the noise-reinforced Brownian motion $(B^p_t)_{t\in\mathbb{R}_+}$ defined by
    \[
        B^p_t\coloneqq t^p\int_0^t s^{-p}\,dB_s,\qquad\text{with covariance}\quad \mathbb{E}[B^p_sB^p_t] = (1-2p)^{-1}s^{1-p}t^p,\qquad\forall~0\leq s\leq t,
    \]
    where $-\infty<p<1/2$ is called the reinforcement strength. Notice that \cite[Eqn. (3)]{Bertoin} $B^p_t$ has the same distribution as $(1-2p)^{-1/2}t^p B_{t^{1-2p}}$ for all $t\in\mathbb{R}_+$ and solves the SDE $dB^p_t=pt^{-1}B^p_t\,dt+\,dB_t$, $B^p_0=B_0=0$, but does not satisfy the time-reversal invariance , unlike standard Brownian motion. The novelty here is to conceive the noise-reinforced SLE driven by $(B^p_t)_{t\in\mathbb{P}_+}$ in the fashion
    \begin{equation}\label{eqn: Z_t noise-reinforced}    
        dZ_t = -\frac{2}{Z_t}\,dt + \sqrt{\kappa}\,dB^p_t = \big(\sqrt{\kappa}t^{-1}pB^p_t-\frac{2}{Z_t}\big)\,dt + \sqrt{\kappa}\,dB_t,\qquad\text{with}\quad Z_0\in\mathbb{H},
    \end{equation}
    modified from (\ref{eqn: SDE of Z_t}). Given that the reinforced mechanism introduces long-term memory on $(B^p_t)_{t\in\mathbb{R}_+}$, what will be the distortion effect on its driven Loewner curve? The geometry of which is known very little to us. It is our belief that the splitting algorithm should numerically predict statistical phenomena which leads to deeper exploration.\par
    Now, to initial the splitting algorithm for the noise-reinforced SLE, instead of replacing by the time-inhomogeneous field $L_0^p:z\in\mathbb{H}\mapsto\sqrt{\kappa}pB^p_t/t-2/z$ at $t\in[0,T]$, we choose to substitute the stochastic differential $dB_t$ with $dB^p_t$ and whence preserves the time-homogeneous field $L_0:z\mapsto-2/z$. Then with a suitable $y_N\to0+$, the splitting algorithm produces the discrete points $(\widetilde{W}^p_{t_k})_{0\leq k\leq N}$ according to 
    \[
        \widetilde{W}^p_{t_{k+1}} = \exp\big(\tfrac{1}{2}h_kL_0\big)\exp\big(B^p_{t_k,t_k+h_k}L_1\big)\exp\big(\tfrac{1}{2}h_kL_0\big)\widetilde{W}^p_{t_k},\qquad\forall~k=0,\ldots,N-1,
    \]
    which subsequently stretches to a continuous approximation $(\widetilde{W}^p_{t})_{t\in[0,T]}$ written out in integral form
    \[
        \widetilde{W}^p_t = \xi + \frac{1}{2}\int_0^t L_0(\widetilde{W}^{(0)}_{p,s}) \,ds + \int_0^t L_1(\widetilde{W}^{(1)}_{p,s})\,dB_s +  \frac{1}{2}\int_0^t L_0(\widetilde{W}^{(2)}_{p,s}) \,ds,\qquad\forall~t\in[0,T],
    \]
    where the increments $\widetilde{W}^{(0)}_{p,t}$ is defined similarly as $\widetilde{W}^{(0)}_{t}$ but with respect to $B^p_t$ in Definition \ref{def: N-V splitting algorithm}, and similarly for $\widetilde{W}^{(1)}_{p,t}$ and $\widetilde{W}^{(2)}_{p,t}$ on $t\in[0,T]$. Simulations of the noise-reinforced SLE via splitting algorithm can be found in Section \ref{sec: Conclusion remarks}, where we shall also discuss their numerical consequences.\par
    \begin{figure}[t!]
    \centering
    \begin{subfigure}[t]{0.45\textwidth}
        \centering
        \includegraphics[width=\textwidth]{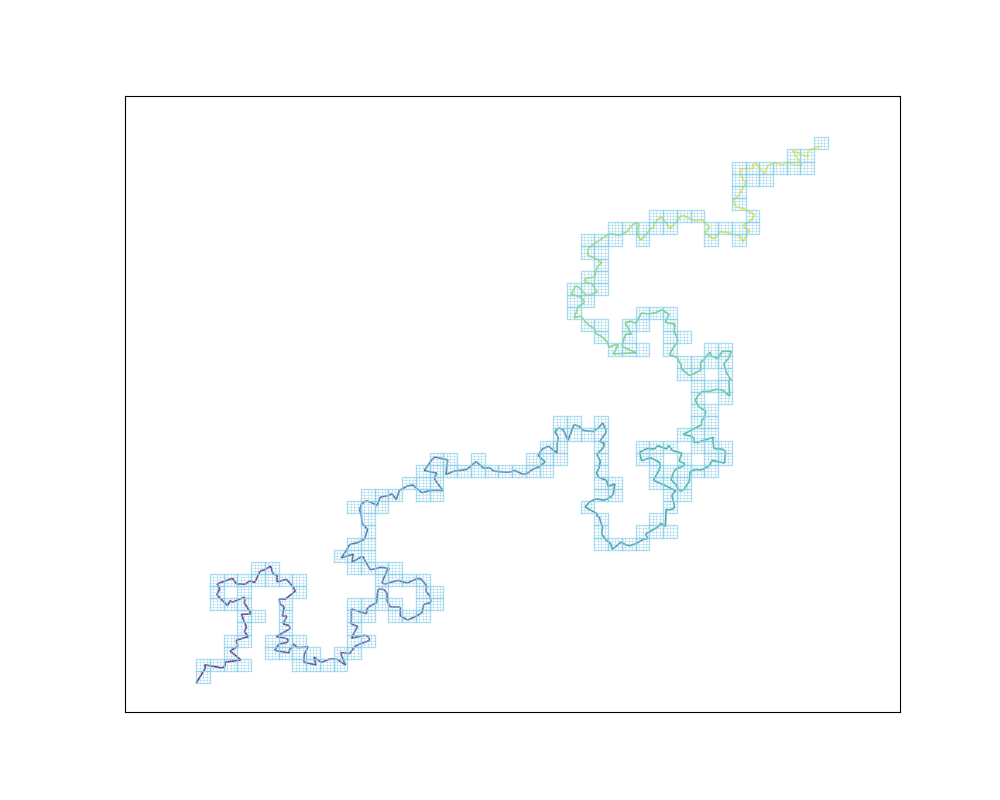} 
        \label{SLE_boxcounting}
    \end{subfigure}
    \hspace{0.0\textwidth} 
    \begin{subfigure}[t]{0.45\textwidth}
        \centering
        \includegraphics[width=\textwidth]{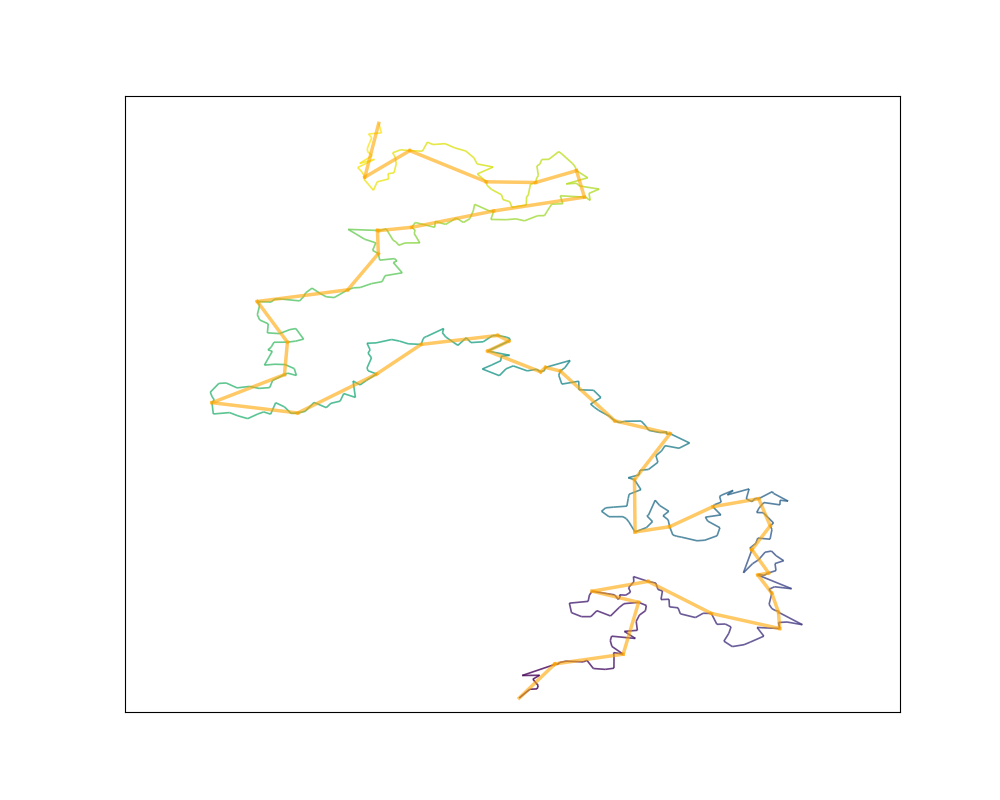} 
        \label{SLE_yardstick}
    \end{subfigure}
    \vspace{-10pt} 
    \caption{Splitting drawing of SLE(4) with box-counting and yardstick illustrations for computing fractal dimensions.}\label{fig: box-counting and yardstick}
\end{figure}
    Despite the rich but mostly unknown geometric features of the Loewner chains driven by semimartingale in the form (\ref{eqn: Z_t noise-reinforced}), these processes are not scale-invariant in general. Recent striking advances into the realm of non-semimartingale driving forces, e.g.~\cite{Tizdast/Ebadi/Cheraghalizadeh/Najafi/Andrade/Herrmann}, has revealed the perseverance of scale-invariance with properly modified deterministic forces. One illustrative example is the \textit{fractional Brownian motion}. Denoted by $(B^H_t)_{t\in\mathbb{R}_+}$, the fractional Brownian motion is written out in fractional integral \cite{Reed/Lee/Truong} for any $0\leq s\leq t$,
    \[
        B^H_t\coloneqq\frac{1}{\Gamma(H+1/2)} \int_0^t (t-s)^{H-1/2}\,dB_s,\qquad\text{with covariance}\quad \mathbb{E}[B^H_s B^H_t] = \big(\abs{s}^{2H}+\abs{t}^{2H}-\abs{t-s}^{2H}\big),
    \]
    with Hurst exponent $0<H<1$ and is revealed by Mandelbrot and van Ness \cite{Mandelbrot/van Ness} to have stationary increments.\par
    We also mention that $(B^H_t)_{t\in\mathbb{R}_+}$ is neither a Markov process nor a semimartingale \cite{Shevchenko}, but significantly it admits time-reversal symmetry \cite{Darses}, i.e.~$B^H_{T-t}-B^H_T$ is distributed identically as $B^H_t$ for all $t\in[0,T]$. In light of this, we are able to modify (\ref{eqn: h_t, reverse Loewner equation}) for the fractional SLE characterized by
    \[
        \partial_t h_t(z) = |h_t(z)-\kappa^HB^H_t|^{2-\frac{1}{H}}\vdot\frac{-2}{h_t(z)-\kappa^H(B^H_{T-t}-B^H_T)},\qquad\forall~z\in\mathbb{H},
    \]
    up to time horizon $T>0$. This form was first proposed in \cite[Eqn. (14)]{Tizdast/Ebadi/Cheraghalizadeh/Najafi/Andrade/Herrmann} which novelly takes into consideration the scale-invariance of $(h_t)_{t\in[0,T]}$. Observe also that this expression can be equivalently written into the Loewner flow driven by $(B^H_t)_{t\in\mathbb{R}_+}$ as
    \begin{equation}
        dZ_t = |Z_t|^{2-\frac{1}{H}}\vdot\frac{-2}{Z_t}\,dt + \kappa^H\,dB^H_t,\qquad\text{with}\quad Z_0\in\mathbb{H}.
    \end{equation}
    Like the noise-reinforced SLE, the theoretical aspects of the fractional SLE remain poorly understood. Fortunately, the pioneering work of \cite{Tizdast/Ebadi/Cheraghalizadeh/Najafi/Andrade/Herrmann} has simulated these Loewner traces, observing their self-similarity, fractal dimension via a yardstick method \cite{Weiss}, and how they are influenced by the values of $H$ and $\kappa$.\par
    In this work, we confirm their numerical observations via the alternative splitting method in Definition \ref{def: N-V splitting algorithm}. With properly selected $y_N\to0+$, we get a set of discrete points $(\widetilde{W}^H_{t_k})_{0\leq k\leq N}$ by
    \[
        \widetilde{W}^H_{t_{k+1}} = \exp\big(\tfrac{1}{2}h_kL_0\big)\exp\big(B^H_{t_k,t_k+h_k}L_1\big)\exp\big(\tfrac{1}{2}h_kL_0\big)\widetilde{W}^H_{t_k},\qquad\forall~k=0,\ldots,N-1,
    \]
    where $L_1:z\mapsto\kappa^H B^H_t$ and $L_0:z\mapsto\abs{z}^{2-\frac{1}{H}}\vdot(-2/z)$ is no longer holomorphic if $H\neq1/2$. These discrete points then subsequently stretches to a continuous process $(\widetilde{W}^H_t)_{t\in[0,T]}$ with integral form
    \[
        \widetilde{W}^H_t = \xi + \frac{1}{2}\int_0^t L_0(\widetilde{W}^{(0)}_{s,H}) \,ds + \int_0^t L_1(\widetilde{W}^{(1)}_{s,H})\,dB_s +  \frac{1}{2}\int_0^t L_0(\widetilde{W}^{(2)}_{s,H}) \,ds,\qquad\forall~t\in[0,T],
    \]
    Here the increments $\widetilde{W}^{(k)}_{s,H}$ are defined similarly as in Definition \ref{def: N-V splitting algorithm}, but with respect to new driving process $(B^H_t)_{t\in[0,T]}$.\par
    Our numerical observations are deferred to Section \ref{sec: Conclusion remarks} where we record the splitting drawings of the noise-reinforced SLE and the fractional SLE curves. As significant statistical phenomena, the fractal dimension of the fractional SLE was first computed in  \cite{Tizdast/Ebadi/Cheraghalizadeh/Najafi/Andrade/Herrmann} and now confirmed in this work. Instead of the yardstick method therein, we adopt the classical box-counting method \cite{Tanaka/Kayama/Kato/Ito} in the context of splitting algorithm. An illustration can be found in Figure \ref{fig: box-counting and yardstick} to schematic difference between these two methods. See also \cite{Chen1,Chen2} for schematic illustrations.

\section{Various convergence analysis of the discrete splitting algorithm}\label{sec: convergence of splitting algorithm to SLE} The natural regularity of SLE dynamics renders us sufficient tools to reveal the convergence of the splitting algorithm in various topologies. Specially, we will show the splitting $\widetilde{Z}(iy_N)$ converges to the random Loewner curve, respectively, in sup-norm and $L^p$ locally on $[0,T]$ in high probability.

\subsection{Local sup-norm convergence in probability.}
    We first look at the sup-norm convergence of the splitting algorithm. An increasing sequence $\phi:\mathbb{N}_+\to\mathbb{R}_+$ is called subpower if $\lim_{N\nearrow\infty}N^{-\nu}\phi(N)=0$ for all $\nu>0$. The proof of Theorem \ref{thm: sup-norm convergence} relies on the following propositions.

    \begin{proposition}\label{prop: estimate event B1}
        \normalfont
        Let $0<\beta<1$. There exists a subpower sequence $(\phi_N)_{N\geq1}$ such that if $\abs{\mathcal{D}_N}\leq 1/N$ and if we write the event
        \[
            \mathcal{B}^{(1)}_N\coloneqq\big\{\normy{\eta-\sum_{k<N}\eta_{t_k}\mathbbm{1}_{[t_k,t_k+h_k)} }_T\leq \tfrac{2\phi(N^{1/2})}{(1-\beta)N^{(1-\beta)/2}} \big\},
        \]
        where $\phi(N^{1/2})\coloneqq\phi_k$ for $k\leq N^{1/2}<k+1$, then $1-\mathbb{P}(\mathcal{B}^{(1)}_N)\leq C N^{-c}$ for some constants $c,C>0$.
    \end{proposition}
    \begin{proof}
        When there is no confusion, we write $t_k\coloneqq t_k(t)$. It is clear that if we replace all $\eta$'s in $\mathcal{B}^{(1)}_N$ by $\gamma$, its probability stays the same. Invoking \cite[Theorem 3.4.2]{Lawler/Limic} according to Ref. \cite{Lawler/Limic} , we find constants $C,C^\prime>0$ such that
        \begin{equation}\label{eqn: event 1 for B1}
            \mathbb{P}\big(\text{osc}( (\sqrt{\kappa}B_t)_{t\in[0,T]},1/N) \leq C\sqrt{\log N/N} \big) \geq 1- C^\prime N^{-2}.
        \end{equation}
        In light of \cite[Eqn. (21)]{Tran} we can find constants $c,c^\prime,C^{\prime\prime}>0$ and $1<\beta<1$ such that
        \begin{equation}\label{eqn: event 2 for B1}
            \mathbb{P} \big(|\partial_z\Tilde{g}^{-1}_t(iy)|\leq cy^{-\beta},\;\forall~t\in[0,1],~y\in[0,1/N^{1/2}]\big)\geq 1- C^{\prime\prime}N^{-c^\prime}.
        \end{equation}
        By \cite[Lemma 2.5]{Tran}, there exists a subpower sequence $(\phi_N)_{N\geq1}$ such that on the intersection of the two events in (\ref{eqn: event 1 for B1}) and (\ref{eqn: event 2 for B1}), we have
        \[
            \phi(N^{1/2})^{-1}\abs{\gamma_{t}-\gamma_{t_k}} \leq 2\int_0^{1/N^{1/2}} |\partial_z\Tilde{g}_t^{-1}(ir)|\,dr \leq \frac{2}{1-\beta}N^{2/(1-\beta)}
        \]
        with $0<\beta<1$. Hence, the complement of $\mathcal{B}^{(1)}_N$ shrinks polynomially and the assertion is verified.
    \end{proof}

    \begin{proposition}\label{prop: estimate event B2}
        \normalfont
        Let $0<\delta<1$ and $\ell_N\coloneqq N^{(1-\delta)/4}$. If we write the event
        \[
            \mathcal{B}^{(2)}_N\coloneqq\big\{ \normy{Z(iy_N)-\eta}_T\leq N^{-(1-\delta)/4} \big\},
        \]
        then there exists a decreasing sequence $\epsilon_N\to0+$ such that $1-\mathbb{P}(\mathcal{B}^{(2)}_N)\leq\epsilon_N$.
    \end{proposition}
    \begin{proof}
        Following \cite[Lemma 6.7]{Foster/Lyons/Margarint}, we find some $0<\delta<1$ such that $\sup_{0\leq t\leq T}\abs{Z_t(iy_N)-\eta_t}\leq C(\omega)y_N^{1-\delta}$, $\mathbb{P}$-a.s. for some $C(\omega)>0$. Here, this $C(\omega)$ is taken to be the minimal random positive number that admits this inequality for all $N\in\mathbb{N}_+$. Whence we have $\{C(\omega)\leq\ell_N\}\subseteq\mathcal{B}^{(2)}_N$; And by choosing $\epsilon_N\coloneqq\mathbb{P}(C(\omega)\geq\ell_N)$, the assertion is verified in light that $\ell_N\nearrow\infty$ monotonically.
    \end{proof}

    On $\mathcal{B}^{(1)}_N\cap\mathcal{B}^{(2)}_N$, there are $\sup_{0\leq t\leq T}\abs{\eta_t-\eta_{t_k(t)}}\leq \frac{2\phi(N^{1/2})}{(1-\beta)N^{(1-\beta)/2}}$ and $\sup_{0\leq t\leq T}\abs{Z_t(iy_N)-\eta_t}\leq N^{-(1-\delta)/4}$. Combing those two upper-bounds, one observes 
    \begin{equation}
        \normy{Z(iy_N)-\sum_{k<N}Z_{t_k}(iy_N)\mathbbm{1}_{[t_k,t_k+h_k)}}_T\leq\frac{2}{N^{(1-\delta)/4}} + \frac{2\phi(N^{1/2})}{(1-\beta)N^{(1-\beta)/2}},\qquad\forall~N\geq1.
    \end{equation}
    Inspecting Definition \ref{def: N-V splitting algorithm} in the context of (\ref{eqn: SDE of Z_t}), we observe that $\widetilde{Z}_{t_{k+1}}^2+2h_k = ((\widetilde{Z}_{t_k}^2-2h_k)^{1/2}+\sqrt{\kappa}B_{t_k,t_k+h_k})^2$ for each $ k<N$. Whence the Loewner evolution $\widetilde{Z}_{t_k}\mapsto\widetilde{Z}_{t_k+h_k}$ consists of two reversed Loewner flows with an intermediate parallel translation.\par
    In light of \cite[Section 2.]{Kager/Nienhuis/Kadanoff}, a constant force $t\mapsto A$ drives the Loewner chain $g_t(z)=A+( (z-A)^2 + 4t )^{1/2}$ on $[0,T]$ with its reversal $h_T(z) = A + ( (z-A)^2-4T )^{1/2}$ at time $T$. If we use $\iota^\prime_{k,1}$ and $\iota^{\prime\prime}_{k,2}$ for the reverse Loewner flows for the constant force $t\mapsto0$ on $[t_k,t_k+h_k/2]$ and respectively the constant force given by the corresponding value of the Brownian path at the final moment of $[t_k+h_k/2,t_{k+1}]$; And $\iota_{k,2}:t\mapsto t+\sqrt{\kappa}B_{t_k,t_k+h_k}$. Then $\widetilde{Z}_{t_k+h_k}(iy_N)=\iota^{\prime\prime}_{k,1}\circ\iota_{k,2}\iota^\prime_{k,1}\widetilde{Z}_{t_k}(iy_N)$ for each $k<N$.

    \begin{proposition}\label{prop: estimate event B3}
        \normalfont
        If further $\abs{\mathcal{D}_N}\leq(4N+1)^{-3}$ and if we write the perturbation event
        \[
            \mathcal{B}^{(3)}_N\coloneqq\big\{ \normy{ \sum_{k<N}\big( Z_{t_k}(iy_N) - \widetilde{Z}_{t_k}(iy_N) \big)\mathbbm{1}_{[t_k,t_k+h_k)} }_T \leq (4N+1)^{-1/2} \big\},
        \]
        then $1-\mathbb{P}(\mathcal{B}^{(3)}_N)\leq Ce^{-cN}$ for some constants $c,C>0$.
    \end{proposition}
    \begin{proof}
        Define the càdlàg driving force $\Tilde{\lambda}_t\coloneqq \sum_{k<N} \sqrt{\kappa}B_{t_k} \mathbbm{1}_{[t_k-h_k/2,t_k+h_k/2)}(t) $ on $[0,T]$. Viewed as a coarse interpolation to the trace $\sqrt{\kappa}B[0,T]$, $(\Tilde{\lambda}_t)_{t\in[0,T]}$ initiates the Loewner flow $(\widetilde{Z}^\dagger_t(iy_N))_{t\in[0,T]}$ with the form
        \begin{equation}\label{eqn: SDE of cadlag Z tilde}
            \widetilde{Z}^\dagger_t(iy_N) = \Tilde{h}_t(iy_N)-\Tilde{\lambda}_t,\qquad\text{with}\quad \partial_t \Tilde{h}_t(z) = \frac{-2}{\Tilde{h}_t(z)- (\Tilde{\lambda}_{T-t}-\Tilde{\lambda}_T) },\qquad\forall~t\in[0,T].
        \end{equation}
        The above observation compels us to estimate
        \begin{equation}\label{eqn: difference between Z and Z tilde}
            \big| Z_t(iy_N) - \widetilde{Z}^\dagger_t(iy_N) \big| \leq \big| h_t(iy_N) - \Tilde{h}_t(iy_N) \big| + \sup_{0\leq t\leq T}\big| \sqrt{\kappa}B_t - \Tilde{\lambda}_t \big|.
        \end{equation}
        To this end, we write $H_t\coloneqq h_t(iy_N) - \Tilde{h}_t(iy_N)$ for all $t\in[0,T]$. Invoking (\ref{eqn: SDE of Z_t}) and (\ref{eqn: SDE of cadlag Z tilde}), we get
        \[
            \frac{d}{dt}H_t - H_t\zeta(t) = \big(\sqrt{\kappa}B_t - (\Tilde{\lambda}_{T-t}-\Tilde{\lambda}_T) \big)\zeta(t),\quad \zeta(t)\coloneqq \big( h_t(iy_N) - \sqrt{\kappa}(B_{T-t}-B_T) \big)^{-1} \big( \Tilde{h}_t(iy_N) - (\Tilde{\lambda}_{T-t}-\Tilde{\lambda}_T) \big)^{-1}.
        \]
        If we integrate the above differential, we get
        \[
            H_t u(t) = H_0 - \int_0^t (\sqrt{\kappa}B_{T-s}-\Tilde{\lambda}_{T-s}+\Tilde{\lambda}_T-\sqrt{\kappa}B_T) u(s)\zeta(s)\,ds,\qquad\text{with}\quad u(t)\coloneqq e^{-\int_0^t \zeta(s)\,ds },\qquad H_0=0.
        \]
        Henceforth, we obtain
        \begin{equation*}\begin{aligned}
            &\big| h_t(iy_N) - \Tilde{h}_t(iy_N) \big| \leq \int_0^t | \sqrt{\kappa}B_{T-s}-\Tilde{\lambda}_{T-s}+\Tilde{\lambda}_T-\sqrt{\kappa}B_T | e^{\int_s^t\Re\zeta(\tau)\,d\tau}|\zeta(s)|\,ds\\
            &\qquad\leq \sup_{0\leq\tau\leq T} |\sqrt{\kappa}B_{\tau}-\Tilde{\lambda}_{\tau}| \int_0^t e^{\int_s^t\Re\zeta(r)\,dr}|\zeta(s)|\,ds \leq \sup_{0\leq\tau\leq T} |\sqrt{\kappa}B_{\tau}-\Tilde{\lambda}_{\tau}| \big( e^{\int_s^t|\zeta(r)|\,dr}-1 \big),
        \end{aligned}\end{equation*}
        where the last inequality is due to \cite[Lemma 2.3, Eqn. (2.12)]{Viklund/Rohde/Wong}. Furthermore, \cite[Eqn. (2.12)]{Viklund/Rohde/Wong} yields $\int_0^t\abs{\zeta(s)}\,ds$ is less than or equal to $\frac{1}{2}\log N(4+y_N^2)$. Thus, looking back to (\ref{eqn: difference between Z and Z tilde}), we have 
        \begin{equation}\label{eqn: further difference}
            \big|Z_t(iy_N)-\widetilde{Z}^\dagger_t(iy_N)\big| \leq N^{1/2}(4+y_N^2)^{1/2} \sup_{0\leq t\leq T} |\sqrt{\kappa}B_t-\Tilde{\lambda}_t| \leq (4N+1)^{1/2} \sup_{0\leq t\leq T} |\sqrt{\kappa}B_t-\Tilde{\lambda}_t|.
        \end{equation}
        Now notice that $\normx{\sqrt{\kappa}B -\Tilde{\lambda} }_T\leq \sup_{0\leq k<N}\normx{\sqrt{\kappa}B_{t_k+\vdot}}_{h_k}$ and that by \cite[Corollary 2.2]{Boukai} we have $\mathbb{P}(S_t\leq x)=2\Phi(t^{-1/2}x)$ where $S_t\coloneqq\sup_{0\leq s\leq t}B_s$ and $\Phi\sim\mathcal{N}(0,1)$ is the cumulant function of a standard Gaussian. Via reflection principle,
        \[
            \mathbb{P}\big( \normx{\sqrt{\kappa}B}_{h_k}\geq (4N+1)^{-1} \big) \leq 2\mathbb{P}\big( S_{h_k}\geq \kappa^{-1/2}(4N+1)^{-1} \big)\leq 2(2/\pi)^{1/2}\exp\big( -\tfrac{1}{2\kappa}h_k^{-1}(4N+1)^{-2} \big).
        \]
        Now that we have chosen $\abs{\mathcal{D}_N}\leq(4N+1)^{-3}$, then on the event $\{\normx{\sqrt{\kappa}B-\Tilde{\lambda}}_T\geq(4N+1)^{-1}\}$, we know $\normx{\sqrt{\kappa}B_{t_k+\vdot}}_{h_k/2}\geq(4N+1)^{-1}$ for some $0\leq k<N$. Hence,
        \[
            \mathbb{P}\big(\normx{\sqrt{\kappa}B-\Tilde{\lambda}}_T\geq(4N+1)^{-1}\big)\leq \sum_{k<N} \mathbb{P}\big(\normx{\sqrt{\kappa}B_{t_k+\vdot}}_{h_k/2}\geq(4N+1)^{-1}\big) \leq 2T(4N+1)^3 e^{-(4N+1)/2\kappa}.
        \]
        Observe that on the complement event of $\{\normx{\sqrt{\kappa}B-\Tilde{\lambda}}_T\geq(4N+1)^{-1}\}$, in light of (\ref{eqn: further difference}), we get $\normx{Z(iy_N)-\widetilde{Z}^\dagger(iy_N)}_T\leq(4N+1)^{-1/2}$ with the probability of this occurrence at least $1 - 2T(4N+1)^3 e^{-(4N+1)/2\kappa}$. Since the splitting output $(\widetilde{Z}_t(iy_N))_{t\in[0,T]}$ coincides with the trace $\widetilde{Z}^\dagger(iy_N)$ at each $t_k\in\mathcal{D}_N$, the assertion is then verified.
    \end{proof}

    Following the Definition \ref{def: N-V splitting algorithm} in the context of SLE, at each $N$th step we could compute $\widetilde{Z}^{(0)}_t=\text{exp}((t-t_k)L_0/2)\widetilde{Z}_{t_k}=(\widetilde{Z}_{t_k}^2-2(t-t_k))^{1/2}$, $\widetilde{Z}_t^{(1)}=\text{exp}(B_{t_k,t}L_1)=(\widetilde{Z}_{t_k}^2-2h_k)^{1/2}+B_{t_k,t}$ and $\widetilde{Z}_t^{(2)}=\text{exp}((t-t_k)L_0/2)\widetilde{Z}^{(1)}_{t_{k+1}}=((\widetilde{Z}^{(1)}_{t_{k+1}})^2-2(t-t_k))^{1/2}$ for any $t_k\leq t<t_k+h_k$ and $0\leq k<N$. Written out in integral form, we get
    \begin{equation}\begin{aligned}\label{eqn: exact from of difference}
        &\widetilde{Z}_t(iy_N)-\widetilde{Z}_{t_k(t)}(iy_N) = \frac{1}{2} \int_{t_k(t)}^t L_0(\widetilde{Z}_s^{(2)}) \,ds + \int_{t_k(t)}^t L_1(\widetilde{Z}_s^{(1)}) \,dB_s +\ frac{1}{2} \int_{t_k(t)}^t L_0(\widetilde{Z}_s^{(0)}) \,ds,\\
        &\qquad = \int_{t_k}^t \sqrt{\kappa} \,dB_s - \int_{t_k}^t \big((\widetilde{Z}_{t_{k}})^2-2(s-t_k)\big)^{-1/2}\,ds - \int_{t_k}^t \big(\big(\widetilde{Z}_{t_{k+1}}^{(1)}\big)^2-2(s-t_k)\big)^{-1/2}\,ds,\qquad\forall~t\in[0,T],
    \end{aligned}\end{equation}
    and can be further expanded if we replace $\widetilde{Z}^{(1)}_{t_{k+1}}$ in terms of $\widetilde{Z}_{t_k}$. This exact form of difference then leads us to the following result.

    \begin{proposition}\label{prop: estimate event B4}
        \normalfont
        Let $\kappa\neq0$ and $\abs{\mathcal{D}_N}\leq(4N+1)^{-3}$. If we write the event
        \[
            \mathcal{B}^{(4)}_N\coloneqq\big\{ \normy{\widetilde{Z}(iy_N)-\sum_{k<N} \widetilde{Z}_{t_k}(iy_N)\mathbbm{1}_{[t_k,t_k+h_k)}}_T \leq 2N^{-1/2}+4N^{-1/4} \big\},
        \]
        then $1-\mathbb{P}(\mathcal{B}^{(4)}_N)\leq Ce^{-cN^{1/2}}$ for some constants $c,C>0$.
    \end{proposition}
    \begin{proof}
        A general result $\Im z\leq\Im (z^2-A)^{1/2}$, $\Im z=\Im z+A$ for all $z\in\mathbb{H}$ and $A\in\mathbb{P}$ follows from \cite[Section 6.1]{Foster/Lyons/Margarint}. Thus, from (\ref{eqn: exact from of difference}) we have
        \[
            \big| \widetilde{Z}_t(iy_N) - \widetilde{Z}_{t_k}(iy_N) \big| \leq |\sqrt{\kappa}B_{t_k,t}| + \frac{2h_k}{\Im \widetilde{Z}_{t_k}(iy_N)} \leq |\sqrt{\kappa}B_{t_k,t}| + 2y_N^{-1}h_k,
        \]
        where the last inequality is due to \cite[Lemma 4.9]{Kemppainen}, which says that the map $t\in[0,T]\mapsto\Im \widetilde{Z}_t(iy_N)$ is increasing. In light of \cite[Corollary 2.2]{Boukai}, we get
        \[
            \mathbb{P}\big( \normx{\sqrt{\kappa}B}_{h_k}\geq N^{-1/4} \big) \leq 2\mathbb{P}\big( S_{h_k}\geq\kappa^{-1/2}N^{-1/4} \big) \leq 2(2/\pi)^{1/2}\exp\big(\tfrac{1}{2\kappa}h_k^{-1}N^{-1/2}\big),\qquad\forall~N\geq1,
        \]
        by the reflection principle. Since $\abs{\mathcal{D}_N}\leq (4N+1)^{-3}$, the assertion then follows.
    \end{proof}
    \begin{proof}[Proof of Theorem \ref{thm: sup-norm convergence}.]
        Combining Propositions \ref{prop: estimate event B1}, \ref{prop: estimate event B2}, \ref{prop: estimate event B3}, \ref{prop: estimate event B4} and on the event $\mathcal{B}^{(1)}_N \cap \mathcal{B}^{(2)}_N \cap \mathcal{B}^{(3)}_N \cap \mathcal{B}^{(4)}_N$, with
        \[
            \varphi_1(N)\coloneqq 3N^{-(1-\delta)/4} + 2(1-\beta)^{-1}\phi(N^{1/2})N^{-(1-\beta)/2} + (4N+1)^{-1/2} + 2N^{-1/2} + 4N^{-1/4},
        \]
        given $\abs{\mathcal{D}_N}\leq(4N+1)^{-3}$, we have
        \begin{equation*}\begin{aligned}
            &\normy{\eta-\widetilde{Z}(iy_N)}_T \leq \normy{\eta-Z(iy_N)}_T + \normy{Z(iy_N)-\sum_{k<N}Z_{t_k}(iy_N)\mathbbm{1}_{[t_k,t_k+h_k)}}_T\\
            &\qquad+ \normy{\sum_{k<N} \big(Z_{t_k}(iy_N)-\widetilde{Z}_{t_k}(iy_N)\big)\mathbbm{1}_{[t_k,t_k+h_k)} }_T + \normy{\sum_{k<N} \widetilde{Z}_{t_k}(iy_N)\mathbbm{1}_{[t_k,t_k+h_k)} - \widetilde{Z}(iy_N) }_T \leq \varphi_1(N).
        \end{aligned}\end{equation*}
        Choose $\varphi_2(N)\coloneqq \epsilon_N + CN^{-c} + C^\prime e^{-c^\prime N^{1/2}}$ for suitable constants $c,c^\prime,C,C^\prime>0$ determined in Propositions \ref{prop: estimate event B1}, \ref{prop: estimate event B2}, \ref{prop: estimate event B3}, \ref{prop: estimate event B4}. We then have $1-\mathbb{P}(\mathcal{B}^{(1)}_N \cap \mathcal{B}^{(2)}_N \cap \mathcal{B}^{(3)}_N \cap \mathcal{B}^{(4)}_N)\leq\varphi_2(N)\to0+$ as the step $N$ advances, verifying the claim.
    \end{proof}

\subsection{Insights on a version of $L^p$-norm convergence.}
    Our next journey is to view the splitting algorithm through the lens of another metric perspective, namely, the $L^p$-norm topology. The proof of Theorem \ref{thm: sup-norm convergence} relies on the following propositions.

    \begin{proposition}\label{prop: estimate A1}
        \normalfont
        We can find an event $\mathcal{A}^{(1)}_N$ and two decreasing sequences $\Tilde{\psi}^{(1)}_N\to0+$, $\Hat{\psi}^{(1)}_N\to0+$ as $N\to\infty$ so that
        \[
            \mathbb{E}\bigg[\int_0^T \big|\eta_t -Z_t(iy_N)\big|^p\,dt\vdot I_{\mathcal{A}_N^{(1)}}\bigg]\leq\Tilde{\psi}^{(1)}_N\qquad\text{with}\quad\mathbb{P}(\mathcal{A}_N^{(1)})\geq1-\Hat{\psi}^{(1)}_N.
        \]
    \end{proposition}

    \begin{proof}
        By \cite[Lemma 6.7]{Foster/Lyons/Margarint} we find $0<\delta<1$ with $\sup_{0\leq t\leq T}|\eta_t-Z_t(iy_N)|\leq C(\omega)N^{-(1-\delta)/2}$, $\mathbb{P}$-a.s. Take the event $\mathcal{A}^{(1)}_N\coloneqq\{C(\omega)\leq N^{(1-\delta)/4} \}$. Then $\Hat{\psi}^{(1)}_N\coloneqq\mathbb{P}(C(\omega)\geq N^{(1-\delta)/4} )\to0+$ and 
        \[
            \mathbb{E}\big[\int_0^T|\eta_t-Z_t(iy_N)|^p\,dt\vdot I_{\mathcal{A}^{(1)}_N}\big] \leq \mathbb{E}\big[ \normx{\eta-Z(iy_N)}_T^p,\, \mathcal{A}^{(1)}_N \big] \leq N^{-p(1-\delta)/4} \coloneqq\Tilde{\psi}^{(1)}_N\to0.
        \]
        And the assertion is verified.
    \end{proof}

    \begin{proposition}\label{prop: estimate A2}
        \normalfont
        We can find an event $\mathcal{A}^{(2)}_N$ and two decreasing sequences $\Tilde{\psi}^{(2)}_N\to0+$, $\Hat{\psi}^{(2)}_N\to0+$ as $N\to\infty$ so that
        \[
            \mathbb{E} \bigg[ \int_0^T \big| Z_t(iy_N) -Z_{t_k(t)}(iy_N)\big|^{p}\,dt \vdot I_{\mathcal{A}_N^{(2)}} \bigg] \leq \Tilde{\psi}^{(2)}_N\qquad\text{with}\quad\mathbb{P}(\mathcal{A}_N^{(2)})\geq1-\Hat{\psi}^{(2)}_N.
        \]
    \end{proposition}
    \begin{proof}
        Invoking \cite[Propositions 3.2, 4.8]{Viklund/Lawler} according to Ref. \cite{Viklund/Lawler} , we can find $0<\beta<1$ with $\sup_{0\leq t\leq T}|\eta_t-\eta_{t_k(t)}|\leq C^\prime(\omega)N^{-(1-\beta)/2}$, $\mathbb{P}$-a.s. It follows from \cite[Lemma 6.7]{Foster/Lyons/Margarint} that $\sup_{0\leq t\leq T}|Z_t(iy_N)-Z_{t_k(t)}(iy_N)|\leq2C(\omega)N^{-(1-\delta)/2}+C^\prime(\omega)N^{-(1-\beta)/2}$. Now we take $\mathcal{A}^{(2)}_N\coloneqq\mathcal{A}^{(1)}_N\cap\{C^\prime(\omega)\leq N^{(1-\beta)/4}\}$. Then $\Hat{\psi}^{(2)}_N\coloneqq\Hat{\psi}^{(1)}_N+\mathbb{P}(C^\prime(\omega)\geq N^{(1-\beta)/4})\to0+$ monotonically and
        \begin{equation*}\begin{aligned}
            &\mathbb{E} \big[ \int_0^T|Z_t(iy_N)-Z_{t_k(t)}(iy_N)|^{p}\,dt \vdot I_{\mathcal{A}_N^{(2)}} \big] \leq \mathbb{E} \big[ \normx{Z(iy_N) - \sum_{k<N} Z_{t_k}(iy_N)\mathbbm{1}_{[t_k,t_k+h_k)} }_{T}^{p},\, \mathcal{A}_n^{(2)} \big]\\ 
            &\qquad \leq 2^{2p-1} N^{-p(1-\delta)/4} + 2^{p-1} N^{-p(1-\beta)/4} \coloneqq\Tilde{\psi}^{(2)}_N\to0.
        \end{aligned}\end{equation*}
        And the assertion is verified.            
    \end{proof}

    \begin{proposition}\label{prop: estimate A3}
        \normalfont
        If $\abs{\mathcal{D}_N}\leq(4N+1)^{-3}$, we can find an event $\mathcal{A}^{(3)}_N$ and two decreasing sequences $\Tilde{\psi}^{(3)}_N\to0+$ and $\Hat{\psi}^{(3)}_N\to0+$ as $N\to\infty$ so that
        \[
            \mathbb{E} \bigg[ \int_0^T \big| Z_{t_k(t)}(iy_N)-\widetilde{Z}_{t_k(t)}(iy_N)\big|^p\,dt\vdot I_{\mathcal{A}_N^{(3)}} \bigg] \leq \Tilde{\psi}^{(3)}_N\qquad\text{with}\quad\mathbb{P}(\mathcal{A}_N^{(3)})\geq1-\Hat{\psi}^{(3)}_N.
        \]
    \end{proposition}
    \begin{proof}
        Invoking Proposition $\ref{prop: estimate event B3}$ and taking $\mathcal{A}^{(3)}_N\coloneqq\mathcal{B}^{(3)}_N$, we can choose $\Hat{\psi}^{(3)}_N\coloneqq2e^{-\kappa^{-1}(4N+1)}\to0+$. Furthermore,
        \begin{equation*}\begin{aligned}
            &\mathbb{E} \big[ \int_0^T |Z_{t_k(t)}(iy_N)-\widetilde{Z}_{t_k(t)}(iy_N)|^p\,dt\vdot I_{\mathcal{A}_n^{(3)}}\big] \leq \mathbb{E} \big[ \normx{\sum_{k<N} (Z_{t_k}(iy_N)-\widetilde{Z}_{t_k}(iy_N))\mathbbm{1}_{[t_k,t_k+h_k)} }_T^p,\, \mathcal{A}_N^{(3)} \big]\\
            &\qquad \leq (4N+1)^{-p/2} \coloneqq\Tilde{\psi}^{(3)}_N\to0.
        \end{aligned}\end{equation*}
        And the assertion is verified.
    \end{proof}

    To analyse the difference $|\widetilde{Z}_{t_k(t)}-\widetilde{Z}_t|$ for $t_k\leq t<t_k+h_k$, we need some famous results \cite[Section 6.5]{Folland} on the interpolation between Lebesgue function spaces. For any $1<p<r<q$ and $f\in L^p[0,T]\cap L^q[0,T]$, we have $f\in L^r$ with $\normx{f}_{L^r}^{1/p-1/q}\leq \normx{f}_{L^p}^{1/r-1/q}\normx{f}_{L^q}^{1/r-1/q}$. On the other hand, we can express the supremum Brownian motion via Gamma functions \cite[Corollary 2.2]{Boukai} as $\mathbb{E}[\sup_{0\leq s\leq t}\abs{B_s}^{2k}]=\pi^{-1/2}(2t)^{k}\Gamma(k+1/2)$ for all $t\in\mathbb{R}_+$ and $k\geq1$.

    \begin{proposition}\label{prop: estimate A4}
        \normalfont
        If $\abs{\mathcal{D}_N}\leq N^{-1}$, we can find a decreasing sequence $\Tilde{\psi}^{(4)}_N\to0+$ as $N\to\infty$ so that
        \[
            \mathbb{E} \bigg[ \int_0^T \big| \widetilde{Z}_{t_k(t)}(iy_N) -\widetilde{Z}_t(iy_N)\big|^p\,dt \bigg] \leq \Tilde{\psi}^{(4)}_N,\qquad\forall~N\geq1.
        \]
    \end{proposition}
    \begin{proof}
        From Proposition \ref{prop: estimate event B4} we know $|\widetilde{Z}_{t_k}-\widetilde{Z}_t|^p\leq 2|\sqrt{\kappa}B_{t_k,t}|^p+2^{p+1}N^{-p/2}$, $\mathbb{P}$-a.s. Since $p\geq2$, there exists $\ell\geq1$ such that $2\ell\leq p<2(\ell+1)$. Notice also that $t-t_k\leq h_k\leq 1/N$ for each $0\leq k<N$. Then,
        \[
            \mathbb{E}\big[ \sup_{k<N} \normx{\sqrt{\kappa}B_{t_k+\vdot}}_{h_k}^{2\ell} \big] \leq \pi^{-1/2}(2\kappa/N)^{\ell}\Gamma(\ell+1/2),\qquad\forall~\ell\geq1.
        \]
        Hence in light of the Lebesgue interpolation relation, we get
        \begin{equation*}\begin{aligned}
            &\mathbb{E}\big[ \sup_{k<N} \normx{\sqrt{\kappa}B_{t_k+\vdot}}^p_{h_k} \big] \leq \mathbb{E}\big[ \sup_{k<N} \normx{\sqrt{\kappa}B_{t_k+\vdot}}^{2\ell}_{h_k} \big]^{\ell+1-p/2} \mathbb{E}\big[ \sup_{k<N} \normx{\sqrt{\kappa}B_{t_k+\vdot}}^{2(\ell+1)}_{h_k} \big]^{\ell(\ell+1-p/2)/(\ell+1)}\\ 
            &\qquad\leq CN^{-\ell(\ell+1-p/2)}\coloneqq\Tilde{\psi}^{(4)}_N\to0.
        \end{aligned}\end{equation*}
        And the assertion then follows.
    \end{proof}

    \begin{proof}[Proof of Theorem \ref{thm: Lp-norm convergence}.]
        Combining Propositions \ref{prop: estimate A1}, \ref{prop: estimate A2}, \ref{prop: estimate A3}, \ref{prop: estimate A4} and on the event $\mathcal{A}_N\coloneqq \mathcal{A}^{(1)}_N\cap\mathcal{A}^{(2)}_N\cap\mathcal{A}^{(3)}_N$ with $\psi_1(N)\coloneqq (\Tilde{\psi}^{(1)}_N)^{1/p} + (\Tilde{\psi}^{(2)}_N)^{1/p} + (\Tilde{\psi}^{(3)}_N)^{1/p} + (\Tilde{\psi}^{(4)}_N)^{1/p}$, given $\abs{\mathcal{D}_N}\leq(4N+1)^{-3}$ we have
        \begin{equation*}\begin{aligned}                        
        &\mathbb{E} \big[\int_0^T |\eta_t-\widetilde{Z}_t(iy_N)|^p\,dt\vdot I_{\mathcal{A}_N}\big]^{1/p} \leq \mathbb{E}\big[ \int_0^T |\eta_t -Z_t(iy_N)|^{p}\,dt\vdot I_{\mathcal{A}_N^{(1)}}\big]^{1/p} + \mathbb{E}\big[ \int_0^T|Z_t(iy_N)-Z_{t_k}(iy_N)|^{p}\,dt \vdot I_{\mathcal{A}_N^{(2)}} \big]^{1/p}\\ 
        &\qquad+ \mathbb{E}\big[\int_0^T|Z_{t_k}(iy_N)-\widetilde{Z}_{t_k}(iy_N)|^p\,dt\vdot I_{\mathcal{A}_N^{(3)}}\big]^{1/p} + \mathbb{E}\big[\int_0^T|\widetilde{Z}_{t_k}(iy_N)-\widetilde{Z}_t(iy_N)|^p\,dt\big]^{1/p} \leq  \psi_1(N)^{1/p}\to0,
    \end{aligned}\end{equation*}
    where $1-\mathbb{P}(\mathcal{A}_N)\leq\psi_2(N)\coloneqq\Hat{\psi}^{(1)}_N+\Hat{\psi}^{(2)}_N+\Hat{\psi}^{(3)}_N\to0+$. And the assertion is then verified.
    \end{proof}

\subsection{Alternative power-law interpolation of driving forces.}\label{sec: linear interpolation proof}
    The simulation of SLE via splitting algorithm often requires a pathwise realization of the driving process $(\lambda_t)_{t\in\mathbb{R}_+}$. Practically one discretizes the time interval $[0,T]$ and performs interpolation to the traces of $(\lambda_t)_{t\in[0,T]}$. It is then theoretically necessary to show such interpolation is valid, which is the aim of this section.\par
    Inspired by \cite{Tran} where the square-root interpolation was shown effective, we generalize this method to arbitrary power-law interpolation of $(\lambda_t)_{t\in[0,T]}$, including the linear one where the Loewner flow can be computed explicitly \cite{Kager/Nienhuis/Kadanoff}. That said, we only briefly outline the proof and omit some technical details due to the overlap content with the square-root interpolation in \cite{Tran}.\par
    Given any $p>0$, to implement the $p$th power interpolation pathwise to the driving force $(\sqrt{\kappa}B_t)_{t\in[0,T]}$ in SLE at each $N$th step, we use
    \begin{equation}
        \lambda^{(N)}_{p,t}\coloneqq N^p(t-t_k)^p \sqrt{\kappa}(B_{t_k+h_k}-B_{t_k}) + \sqrt{\kappa}B_{t_k},\qquad\forall~t_k\leq t< t_k+h_k.
    \end{equation}
    Applying this $(\lambda^{(N)}_{p,t})_{t\in[0,T]}$ to the forward Loewner equation (\ref{eqn: g_t, forward Loewner equation}), we produce a forward Loewner chain $(g^{(N)}_{p,t})_{t\in[0,T]}$ which is generated by the Loewner curve $\gamma_{p,N}:[0,T]\to\overline{\mathbb{H}}$. As $N$ advances, we get the following convergence assertion.

    \begin{proposition}
        \normalfont
        Let $\gamma[0,T]$ be the forward Loewner curve for $\kappa\neq8$. If the mesh size $\abs{\mathcal{D}_N}\to0+$ with $N\to\infty$ at a rate $\abs{\mathcal{D}_N}=o(N^{-3})$, then
        \[
            \mathbb{P}\big( \normy{\gamma-\gamma_{p,N}}_T\leq\phi_1(N) \big)\geq1-\phi_2(N),\qquad\forall~N\geq1,
        \]
        for some monotonically decreasing $\phi_i:\mathbb{N}_+\to\mathbb{R}_+$ with $\phi_i\to0+$ for each $i=1,2$.
    \end{proposition}
    We do not aim to give a detailed proof here, but a shortened outline will be presented below. Let $f_t$ be the inverse of $g_t$ and similarly $f^{(N)}_{p,t}$ of $g^{(N)}_{p,t}$ on $\mathbb{H}$, and let $\Hat{f}^{(N)}_{p,t}\coloneqq f^{(N)}_{p,t}(\vdot+\lambda^{(N)}_{p,t})$ be its shifting, and $\Hat{f}_t(\vdot)$ similarly defined. Write $G^{(N)}_{p,k}\coloneqq(\Hat{f}^{(N)}_{p,t_k})^{-1}\circ \Hat{f}^{(N)}_{p,t_k+h_k}$, then it is clear that $\Hat{f}^{(N)}_{p,t_k}=G^{(N)}_{p,0}\circ G^{(N)}_{p,1}\circ\cdots\circ G^{(N)}_{p,k-1}$ for each $k\leq N$. We also denote $\gamma^s_{p,N}:t\in[0,T-s]\mapsto g^{(N)}_{p,s}(\gamma_{p,N}(t+s))$ for all $0\leq t\leq T-s$. Hence for any $k<N$,
    \begin{equation}\label{eqn: 3.9}
        \big| \gamma(r+t_k) - \gamma_{p,N}(r+t_k) \big| \leq \big| \gamma(r+t_k) - \gamma(s+t_k) \big| + \big| \Hat{f}_{t_k}(z) - \Hat{f}_{t_k}(w) \big| + \big| \Hat{f}_{t_k}(w) - \Hat{f}^{(N)}_{p,t_k}(w) \big|,
    \end{equation}
    where $1/N\leq r\leq 2/N$, $w=\gamma^{(N)}_{p,t_k}(r)$ and $z=\arg\max\{\Im \gamma^{(N)}_{p,t_k}(s):\,0\leq s\leq 2/N\}$. The first term on the right-hand side is bounded by the uniform continuity of $\gamma[0,T]$ on the event $\mathcal{B}^{(1)}_N$ by (\ref{eqn: event 1 for B1}).\par
    Now we introduce, for each subpower sequence $(\phi_N)_{N\geq1}$, the box
    \[
        A_{N,\theta,\phi}\coloneqq\big\{ z\in\mathbb{H}:\, \abs{\Re z}\leq N^{-1/2}\phi_N,\, \phi_N^{-1}\leq N^{1/2}\Im z\leq\theta \big\}\subseteq\mathbb{H},\qquad\forall~\theta>0.
    \]
    Invoking \cite[Lemma 2.6]{Tran}, for any $z,z^\prime\in A_{N,\theta,\phi}$ and conformal map $f$ on $\mathbb{H}$, we have $\abs{\partial_z f(z)}\leq C\phi_N^c\abs{\partial_z f(\Im z)}$ and $d_{\mathbb{H},\text{hyp}}(z,z^\prime)\leq C(\log\phi_N+1)$ for constants $c,C>0$, where $d_{\mathbb{H},\text{hyp}}(\vdot)$ denotes the hyperbolic distance \cite{Anderson} on $\mathbb{H}$. In light of \cite[Corollary 1.5]{Pommerenke}, we further know $\abs{f(z)-f(z^\prime)}\leq 2\abs{\partial_z f(z)\Im z}\exp(4d_{\mathbb{H},\text{hyp}}(z,z^\prime))$ for any conformal map $f$ on $\mathbb{H}$ and $z,z^\prime\in\mathbb{H}$. The second term on the right-hand side of (\ref{eqn: 3.9}) can then be estimated if we suitably adjust the above arguments in the context of SLE, see also \cite[Section 3]{Tran}.\par
    In this regard, we also invoke \cite[Lemma 2.2]{Viklund/Rohde/Wong} which provides a perturbation estimate to the distance between two reversed Loewner flows provided we know how to compare their respective driving functions. And this yields the desired estimate to the third term on the right-hand side of (\ref{eqn: 3.9}), verifying the main proposition of this section. In practical running of the splitting algorithm, we combine with the $p$th power interpolation pathwise on the driving traces to ease the numerical computation.

\section{Future perspectives and conclusional remarks}\label{sec: Conclusion remarks}
    \begin{figure}[t!]
    \centering
    \begin{subfigure}[t]{0.45\textwidth}
        \centering
        \includegraphics[width=\textwidth]{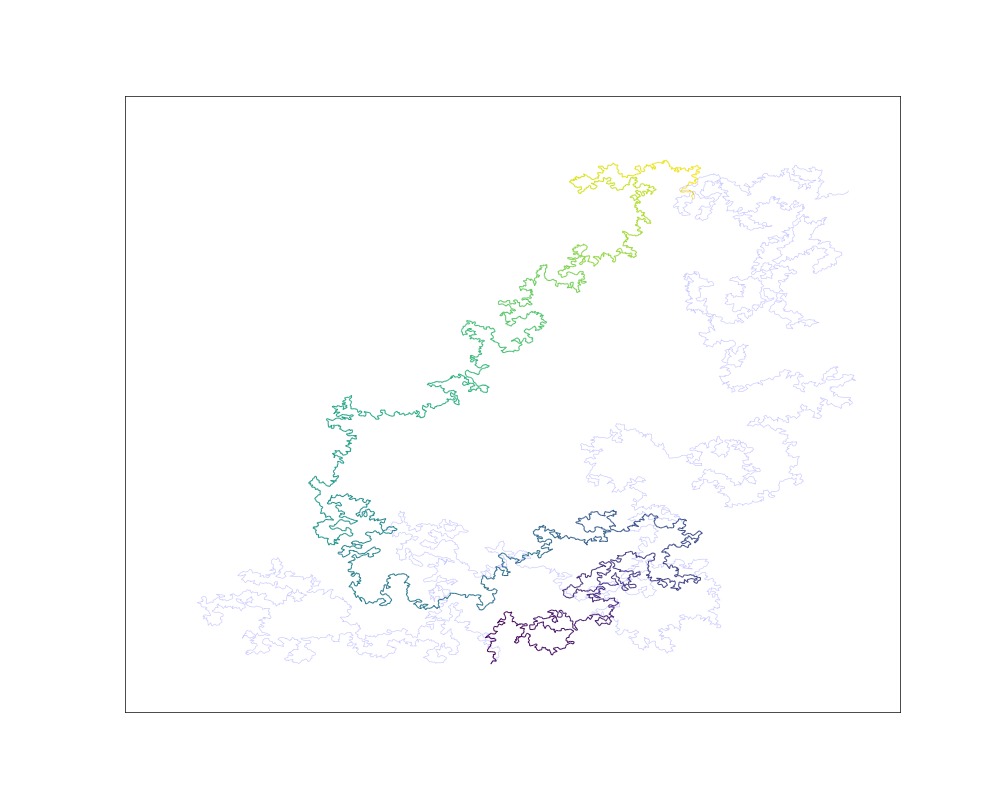} 
        \vspace{-30pt}
        \caption{$p=0.1$}
    \end{subfigure}
    \hspace{0.015\textwidth}
    \begin{subfigure}[t]{0.45\textwidth}
        \centering
        \includegraphics[width=\textwidth]{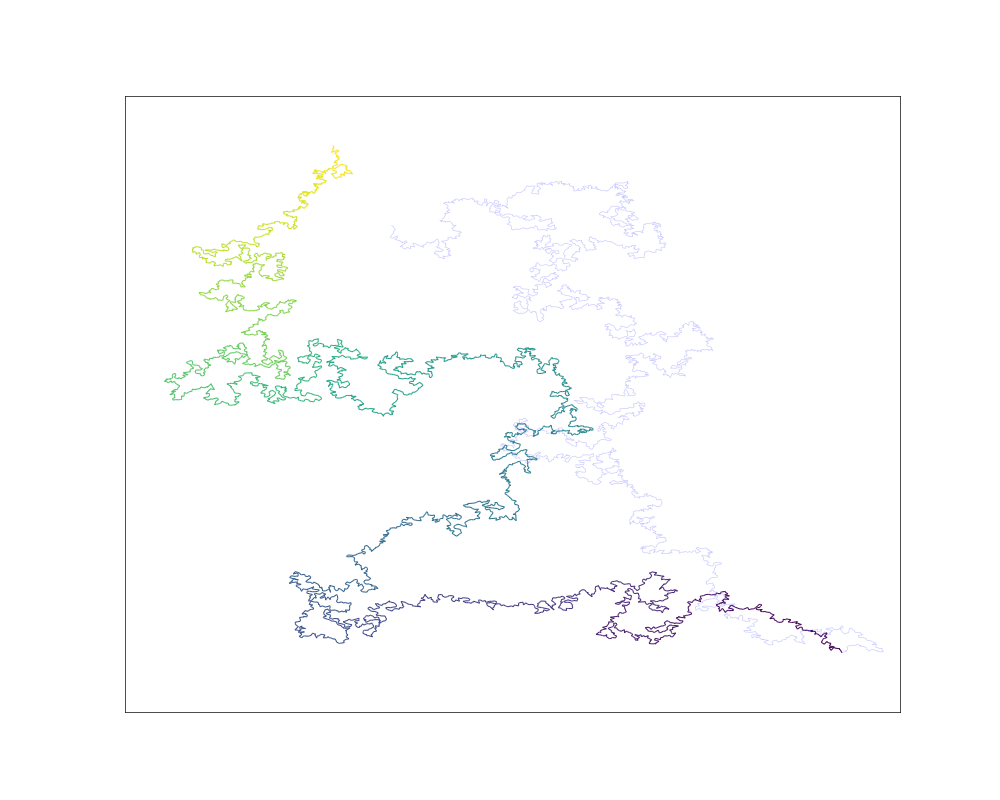} 
        \vspace{-30pt}
        \caption{$p=0.2$}
    \end{subfigure}
    \begin{subfigure}[t]{0.45\textwidth}
        \centering
        \includegraphics[width=\textwidth]{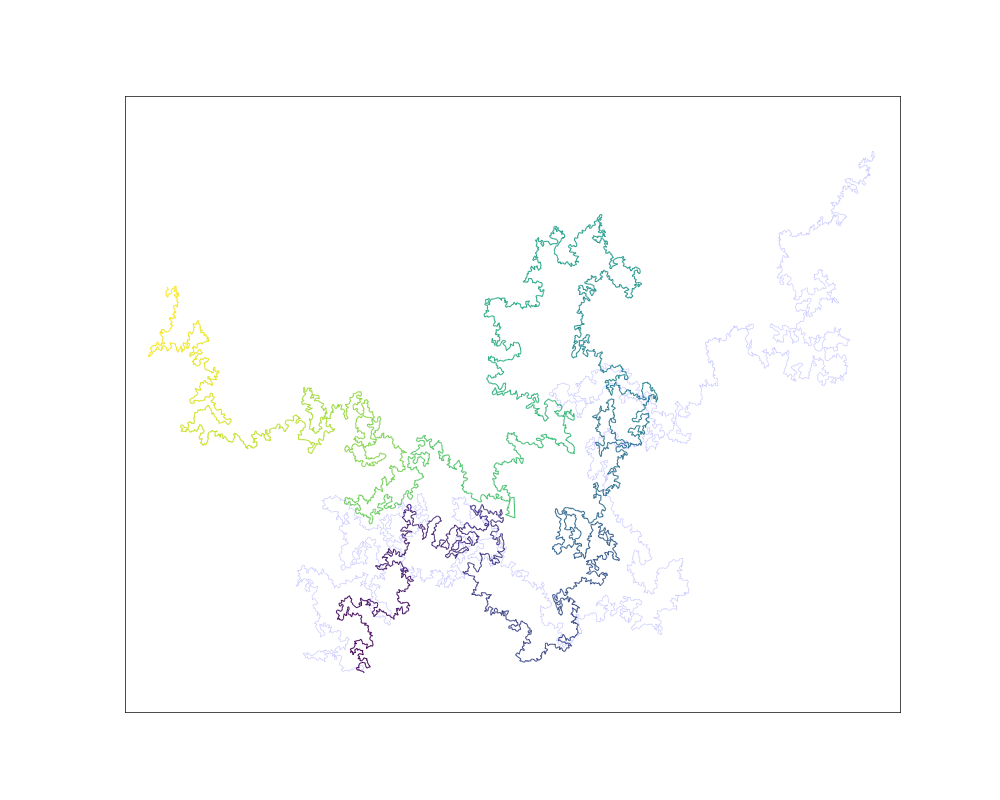} 
        \vspace{-30pt}
        \caption{$p=0.3$}
    \end{subfigure}
    \hspace{0.015\textwidth}
    \begin{subfigure}[t]{0.45\textwidth}
        \centering
        \includegraphics[width=\textwidth]{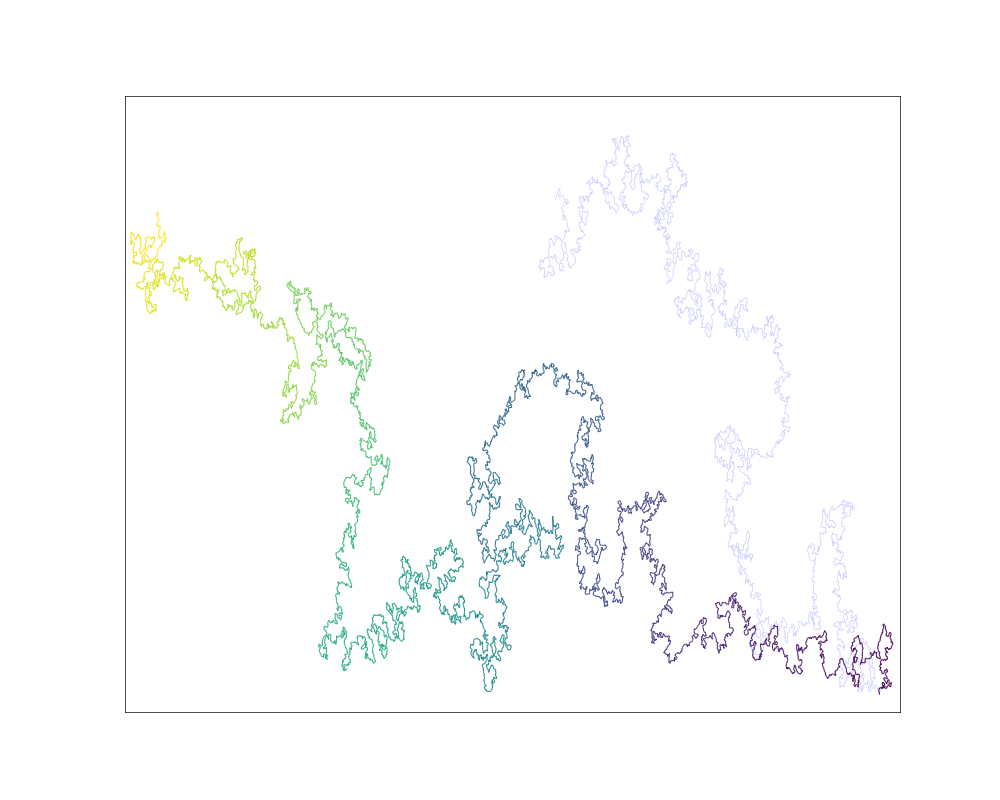} 
        \vspace{-30pt}
        \caption{$p=0.4$}
    \end{subfigure}
    \vspace{9pt}
    \caption{Noise-reinforced SLE(4) with varying reinforcement intensity.}
    \label{fig: noise-reinforced SLE drawings}
\end{figure}

    The aim of this final section is twofold: First, we draw the noise-reinforced SLE and the fractional SLE in Section \ref{sec: Non-standard Loewner chains driven by semimartingales} via the splitting algorithm, before looking at their fractal dimensions; Second, the ingredients for showing the rate of algorithm convergence are briefly outlined which sheds light on a more quantitative analysis of the splitting algorithm.

\subsection{Splitting drawings of noise-reinforced and fractional Loewner curves.} 
    Arising from studies of two-color urn models \cite{Bai/Hu/Zhang} and step-reinforced random walk \cite{Chen/Laulin}, the noise-reinforced Brownian motion $(B^p_t)_{t\in\mathbb{R}_+}$ is characterized by an integral of a reinforced version of white noise. Its time-derivative $dB^p/dt$, viewed in the sense of generalized functions \cite[Chapter 3]{Evans}, has the tendency to repeat itself infinitesimally over time. When the Loewner flow $(Z_t)_{t\in[0,T]}$ is driven by such reinforced forces, one naturally wonders whether its generated Loewner curve inherits such microscopic property. A classical and important subsequent question is then to understand in what sense the local repetition over small time duration disrupts the macroscopic asymptotic SLE behaviors.\par
    To our best knowledge, none of those questions are rigorously answered so far. And as communicated with Nina Holden, such noise-reinforced version of SLE is likely to possess very different regularities compared to the standard Loewner chains. One instance can be found in Figure \ref{fig: noise-reinforced SLE drawings} where the noise-reinforced SLE(4) is less likely to be self-similar but more likely to replicate its previous distortions on time interval $[0,T]$. Nevertheless, its theoretical derivations can be contrastedly less transparent. The Cardy--Smirnov formula of crossing probability for regular SLE \cite{Cardy,Langlands/Pouliot/Saint-Aubin}, e.g.~relies crucially on passing to limits with martingale structure, whereas the semimartingale decomposition of $(B^p_t)_{t\in\mathbb{R}_+}$ induces a non-trivial drift term, forcing us to conceive novel theoretical techniques to resolve this issue.\par
    Very fortunately, effective numerical methods allow us to perceive, at least with high probability, what significant phenomena is going to emerge. For instance, observed from Figure \ref{fig: noise-reinforced SLE drawings} the noise-reinforced Loewner curve is fractal and impossible to be self-similar, which saves great effort if we want to understand its geometry in this aspect. Meanwhile, with gradually intensified memory, i.e.~increasing $p$-value, the noise-reinforced traces look accordingly more twisted, forecasting a possible multiscale phase transition \cite{Aharony,Boettcher} with threshold parameters both $p$ and $\kappa$. Other interesting statistical measures such as the fractal dimension or the winding angle statistics \cite{Najafi/Tizdast/Cheraghalizadeh} can also be numerically quantified based on realizations of the splitting algorithm. In the following, we draw the fractional SLE, which have been proposed in \cite{Tizdast/Ebadi/Cheraghalizadeh/Najafi/Andrade/Herrmann}, and analyse its fractal dimensions.\par
\begin{figure}[t!]
    \centering
    \begin{subfigure}[t]{0.45\textwidth}
        \centering
        \includegraphics[width=\textwidth]{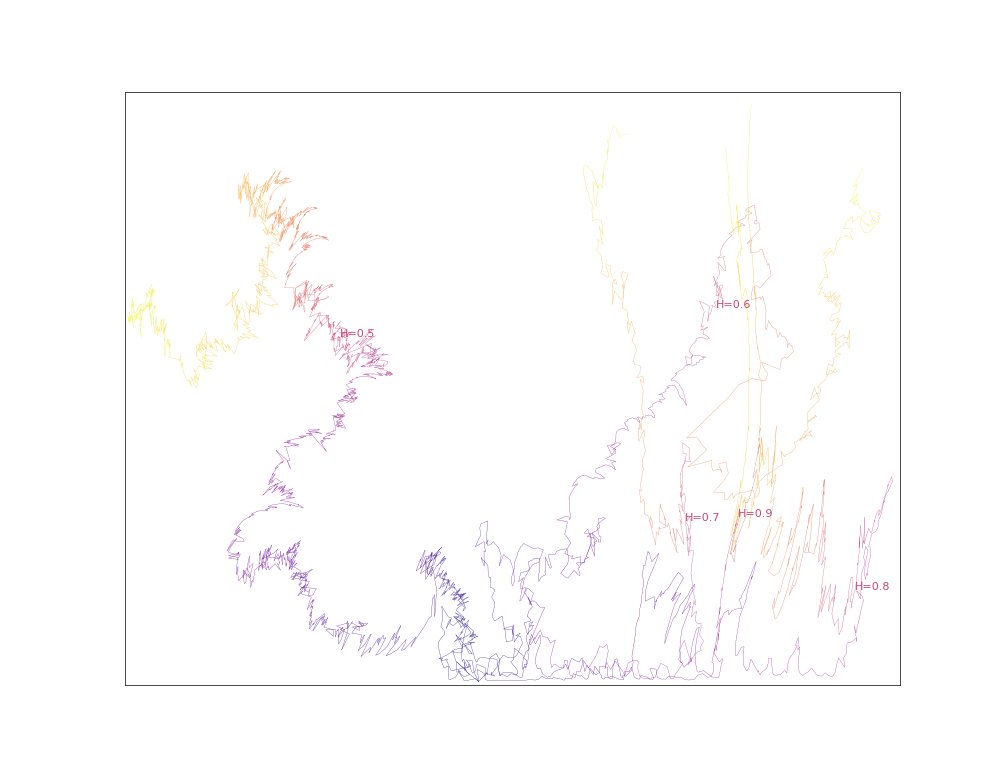} 
        \vspace{-30pt}
        \caption{$\kappa=3$}
    \end{subfigure}
    \hspace{0.015\textwidth}
    \begin{subfigure}[t]{0.45\textwidth}
        \centering
        \includegraphics[width=\textwidth]{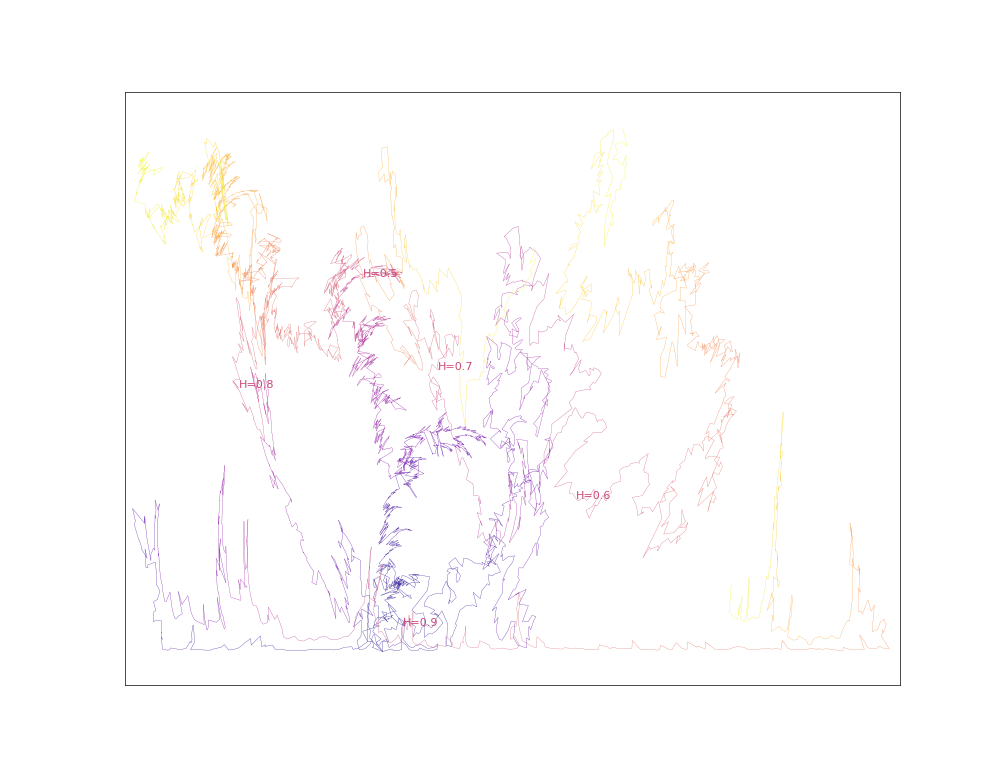} 
        \vspace{-30pt}
        \caption{$\kappa=4$}
    \end{subfigure}
    \begin{subfigure}[t]{0.45\textwidth}
        \centering
        \includegraphics[width=\textwidth]{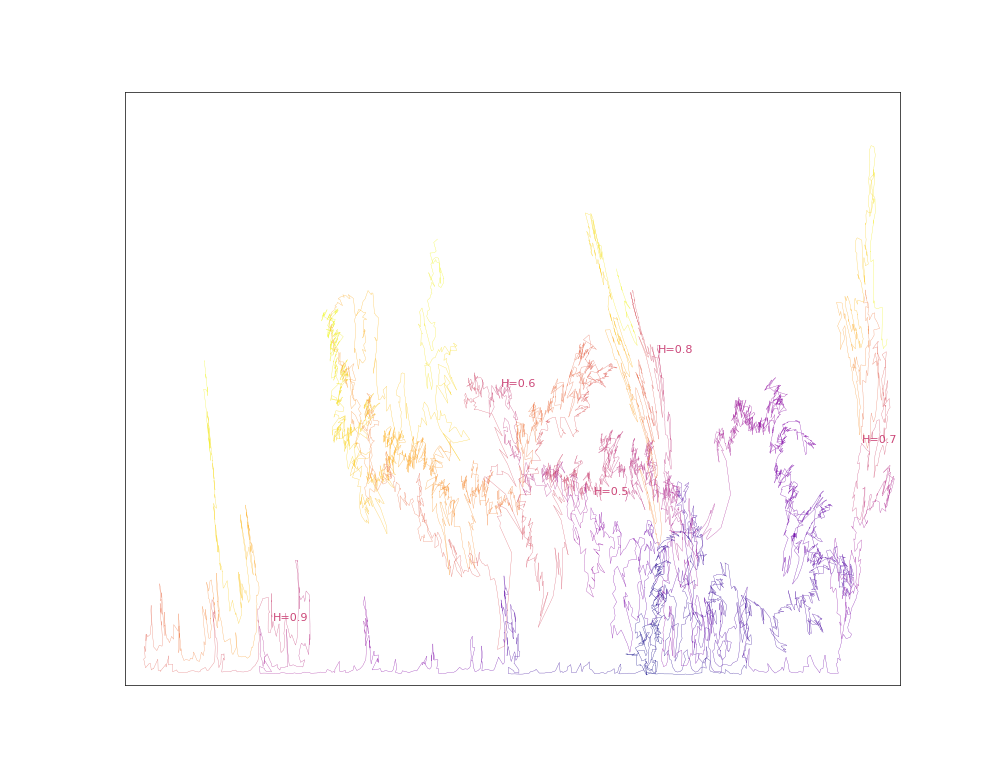} 
        \vspace{-30pt}
        \caption{$\kappa=5$}
    \end{subfigure}
    \hspace{0.015\textwidth}
    \begin{subfigure}[t]{0.45\textwidth}
        \centering
        \includegraphics[width=\textwidth]{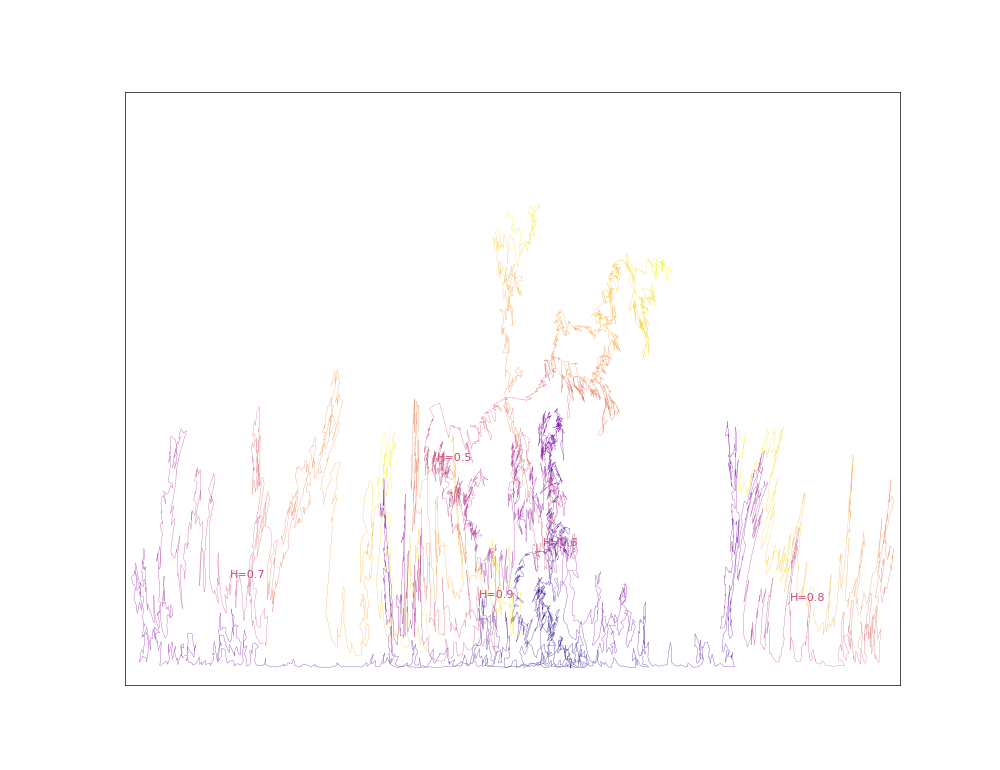} 
        \vspace{-30pt}
        \caption{$\kappa=6$}
    \end{subfigure}
    \vspace{9pt}
    \caption{Fractional SLE$(\kappa)$ with varying diffusivity $\kappa$ and Hurst exponent $H$. Here the Hurst exponent is indicated at the midpoint of each trace.} 
    \label{fig: Fractional SLE}
\end{figure}
    Motivated from the Loewner evolution with subordinated and random time-change \cite{Ghasemi Nezhadhaghighi/Rajabpour/Rouhani}, the study of fractional SLE was initiated by \cite{Tizdast/Ebadi/Cheraghalizadeh/Najafi/Andrade/Herrmann} where their striking finding reveals that suitably modifying the deterministic Loewner force \cite[Eqn. (11)]{Tizdast/Ebadi/Cheraghalizadeh/Najafi/Andrade/Herrmann} should at least numerically yield self-similar traces with the underlying driving force $(B^H_t)_{t\in\mathbb{R}_+}$, whereas the singularity at$r\to0+$ is not of central importance. Although there is no more conformal invariance whenever $H\neq1/2$, the observations in \cite[Section IV]{Tizdast/Ebadi/Cheraghalizadeh/Najafi/Andrade/Herrmann} still shed light on the possibility to describe the interface of variant lattice models \cite{Boyle/Steinhardt}. It is still open, nonetheless, to rigorously attest such microlocal properties of the fractional SLE, despite that numerical simulations should have promised such formulations to be natural.\par
    Taken into account the existing computational uncertainties, we run the simulation of fractional SLE via the alternative splitting algorithm for different $\kappa$ and $H$ values. This algorithm has been shown effective under both the sup-norm and $L^p$ topologies in Section \ref{sec: convergence of splitting algorithm to SLE}, and whose principle is different from the frequently used zipper algorithm \cite{Castro/Lukovic/Pompanin/Andrade/Herrmann}. Not quite surprisingly, we observe exactly like in \cite[Section IV]{Tizdast/Ebadi/Cheraghalizadeh/Najafi/Andrade/Herrmann} that the fractional Loewner curves gets smoother as $H$ increases and coarser when $\kappa$ grows, leading to varying fractal dimensions. This observation can be traced in Figure \ref{fig: Fractional SLE} where the fractal Loewner curves are drawn via the splitting algorithm for various Hurst exponents.\par
    The fractal dimension of random objects is of considerable theoretical importance in that it governs the fashion in which delicate information is propagated as time flows \cite{Herrmann/Stanley}. When the dynamics is restrained by the Hurst exponent $H$ to follow a power-law decay on correlations, various conjectures \cite{Schmittbuhl/Vilotte/Roux,Weinrib/Halperin} and with Fourier filtering \cite{Oliveira/Schrenk/Araujo/Herrmann/Andrade} have proposed explicit relations between $H$ and the fractal dimension. In this work, we use the classical box-counting method to compute such fractional dimension $D_f(\kappa,H)$ of the fractional SLE via splitting. The results in Figure \ref{fig: fSLE fractal dimension} conform to the pioneering observation in \cite[Section IV]{Tizdast/Ebadi/Cheraghalizadeh/Najafi/Andrade/Herrmann} regarding the monotonicity of $D_f(\kappa,H)$ w.r.t. $\kappa$ and $H$. Even so, we have to concede that the combination of box-counting and splitting algorithm does not come without defect. To the aim of computing efficiency, we have adopted the pathwise linear interpolation of fractional driving forces - whose convergence is guaranteed in Section \ref{sec: linear interpolation proof} - leading to small errors since microscopic twistings are neglected. The diagram of $D_f(\kappa,H)$ w.r.t. $\kappa$ and $H$, see Figure \ref{fig: fSLE fractal dimension}, is then non-negligibly above $1$ when $H=1$, differing to \cite[Section IV]{Tizdast/Ebadi/Cheraghalizadeh/Najafi/Andrade/Herrmann}. We thus look forward to upgrade and optimization of the numerical method of calculating $D_f(\kappa,H)$.

\begin{figure}[t!]
    \centering
    \includegraphics[width=0.6\textwidth]{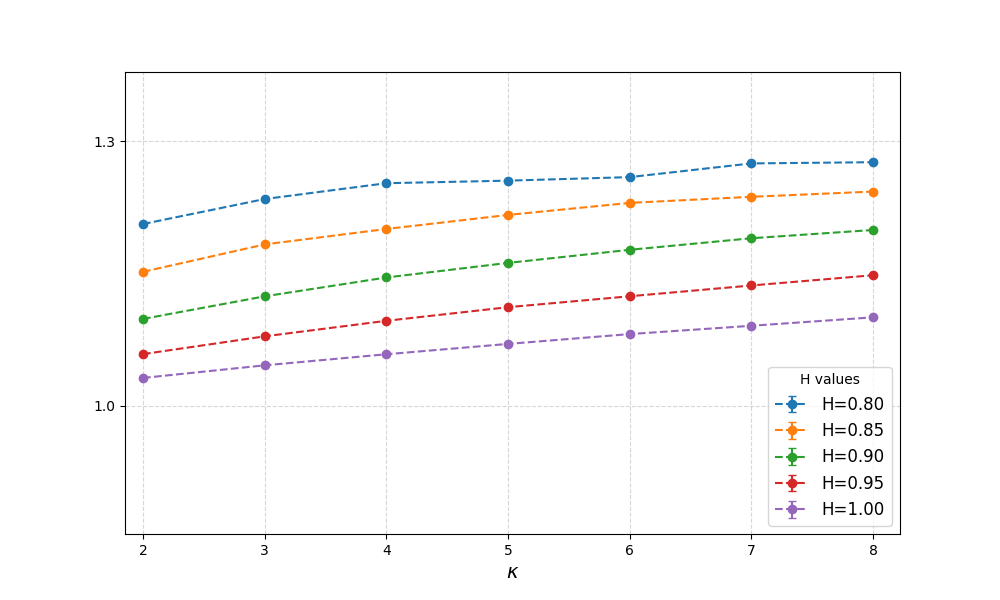} 
    \vspace{-1pt} 
    \caption{Fractal dimension $D_f(\kappa,H)$ of fractional SLE with varying diffusivity $\kappa$ and Hurst exponent $H$.}
    \label{fig: fSLE fractal dimension}
\end{figure}

\subsection{On the quantitative analysis of the velocity of algorithm convergence.} Observing from Figure \ref{fig: box-counting and yardstick} for an underlying drawing of SLE(4) via splitting algorithm, we see that it is a well suited method to simulate analytically tractable field such as SLE. Quoted from Terry Lyons et al. \cite{Foster/Lyons/Margarint}, this splitting algorithm is the first high order numerical method which preserves second and fourth moments of the reverse Loewner flow, prominently distinguishing itself from previous methods by Kennedy \cite{Kennedy} and Tran \cite{Tran}.\par
Hopefully, a precise local error estimate in Section \ref{sec: convergence of splitting algorithm to SLE} would yield the exact rate of convergence of our algorithm. Understanding this velocity is important for the resource allocation and cost control \cite{Dennis/Schnabel}, while a fast splitting algorithm significantly reduces CPU time and memory occupation. We therefore conclude this paper with some new insights on the rate of splitting algorithm convergence to SLE, illuminating the direction of future researches.\par
In the following we formulate the main ideas of giving the precise convergence rate of the splitting algorithm, respectively in the sup-norm and $L^p$ topology scenarios. In Proposition \ref{prop: estimate event B1} the $\normx{\vdot}_T$ norm between $\eta_t$ and $\eta_{t_k(t)}$ on $[0,T]$ is bounded by a subpower function $\phi(N^{1/2})$ with $N\geq1$. One should look to \cite{Lawler/Limic,Tran} for a precise form of $(\phi_N)_{N\geq1}$ and thus sharpen this assertion. The decreasing sequence $(\epsilon_N)_{N\geq1}$ which appeared in Proposition \ref{prop: estimate event B2} remains unexplicitly specified. Indeed, one notices that $\epsilon_N=\mathbb{P}( y_N^{\delta-1} \normx{ Z(iy_N)-\eta }_T\geq N^{(1-\delta)/4} )$ which is characterized purely by the reverse Loewner flow $(h_t)_{t\in[0,T]}$. Consulting the classic SLE literature \cite{Duplantier/Sheffield,Lawler/Sheffield,Werner} should then shed light on the exact order of decay of $(\epsilon_N)_{N\geq1}$.\par
The viewpoint of $L^p$-norm convergence can be treated analogously. To quantify the rate of convergence relies crucially on estimating the sequences $\Hat{\psi}^{(1)}_N=1-\mathbb{P}(\mathcal{A}^{(1)}_N)$ from Proposition \ref{prop: estimate A1} and $\Hat{\psi}^{(2)}_N=1-\mathbb{P}(\mathcal{A}^{(2)}_N)$ from Proposition \ref{prop: estimate A2}. If we take a closer look at those terms, then this is exactly the monotonic sequence $(\epsilon_N)_{N\geq1}$ whose precise value is to be determined. At this moment we should also point out that the presence of the high probability event $\mathcal{A}_N$ from Theorem \ref{thm: Lp-norm convergence} might be unnecessary, i.e.~it might be replaced with $\Omega$. We do not make any predictions here, but such modification definitely requires a finer inspection on the regularities of the reverse Loewner flow $(h_t)_{t\in[0,T]}$, which we shall look at in future projects.

\par\noindent
\textbf{Acknowledgements.} We would like to acknowledge James Foster at Univ.~Oxford for his help in sharing with us his simulation on the traces of SLE(8/3) and SLE(6) via the Ninomiya--Victoir splitting algorithm. We also acknowledge Lukas Schoug at Univ.~Cambridge for his valuable comments and for reading the preliminary version as well as the revised version of our manuscript. It should also not go unmentioned that the very constructive comments from the anonymous reviewers greatly enhance the quality of this work. We also thank Nina Holden for communications on the aspects of noise-reinforced SLE. Numerical drawings in Figure \ref{fig: Fractional SLE} is carried out at the Center of High Performance Computing at NYU Shanghai. Jiaming and Vlad acknowledge supports from the NYU--ECNU Mathematical Institute at NYU Shanghai.

\bibliographystyle{plain}

\begin{thebibliography}{16}
\addcontentsline{toc}{section}{References}

\bibitem{Ableidinger/Buckwar}
M. Ableidinger, E. Buckwar.
\newblock Splitting integrators for the stochastic Landau--Lifshitz equation.
\newblock{\em SIAM J. Sci. Comput.,} 38: 1788--1806, 2016.


\bibitem{Aharony}
A. Aharony.
\newblock Old and new results on multicritical points.
\newblock{\em J. Stat. Phys.,} 110: 659--669, 2003.


\bibitem{Anderson}
J. W. Anderson.
\newblock Hyperbolic Geometry.
\newblock{\em Springer Undergrad. Math. Ser., 1st ed.,} 1999.


\bibitem{Ang/Remy/Sun}
M. Ang, G. Rémy, X. Sun.
\newblock FZZ formula of boundary Liouville CFT via conformal welding.
\newblock{\em J. European Math. Soc.,} to appear, 2023


\bibitem{Bai/Hu/Zhang}
Z.-D. Bai, F. Hu, L.-X. Zhang.
\newblock Gaussian approximation theorems for urn models and their applications.
\newblock{\em Ann. Appl. Probab.,} 12 (4): 1149--1173, 2002.


\bibitem{Bauer}
R. O. Bauer.
\newblock Discrete Löwner evolution.
\newblock{\em Ann. Fac. Sci. Toulouse Math.,} 12 (4): 433--451, 2003.


\bibitem{Beffara}
V. Beffara.
\newblock The dimension of the SLE curves.
\newblock{\em Ann. Probab.,} 36 (4): 1421--1452, 2008.


\bibitem{Beffara2} 
V. Beffara.
\newblock Simulation of SLE.
\newblock{\em GitHub Reposit.,} https://github.com/vbeffara/simulations, 2020.


\bibitem{Bertenghi/Rosales-Ortiz} 
M. Bertenghi, A. Rosales-Ortiz.
\newblock Joint invariance principles for random walks with positively and negatively reinforced steps.
\newblock{\em J. Stat. Phys.,} 189, 35, 2022.


\bibitem{Bertoin0}
J. Bertoin.
\newblock Noise reinforcement for Lévy processes.
\newblock{\em Ann. Inst. Henri Poincaré Probab. Stat.,} 56 (3): 2236--2252, 2020.


\bibitem{Bertoin}
J. Bertoin.
\newblock Universality of noise reinforced Brownian motions.
\newblock{\em In. Out. Equilibrium III, Celebrat. Vladas Sidoravicius., Prog. Probab.,} 77, 2021.


\bibitem{Bertoin2}
J. Bertoin.
\newblock On the local times of noise reinforced Bessel processes.
\newblock{\em Ann. Henri Lebesgue,} 5: 1277--1294, 2022.


\bibitem{Boettcher}
I. Boettcher.
\newblock Scaling relations and multicritical phenomena from functional renormalization.
\newblock{\em Phys. Rev. E,} 91, 062112, 2015.


\bibitem{Boukai}
B. Boukai.
\newblock An explicit expression for the distribution of the supremum of Brownian motion with a change point.
\newblock{\em Commun. Stat. Theor. Method.,} 19: 31--40, 1990.


\bibitem{Boyle/Steinhardt}
L. Boyle, P. J. Steinhardt.
\newblock Self-similar one-dimensional quasilattices.
\newblock{\em Phys. Rev. B,} 106, 144112, 2022.


\bibitem{Buckwar/Samson/Tamborrino/Tubikanec}
E. Buckwar, A. Samson, M. Tamborrino, I. Tubikanec.
\newblock A splitting method for SDEs with locally Lipschitz drift: Illustration on the FitzHugh--Nagumo model.
\newblock{\em Appl. Numer. Math.,} 179: 191--220, 2022.


\bibitem{Buckwar/Tamborrino/Tubikanec}
E. Buckwar, M. Tamborrino, I. Tubikanec.
\newblock Spectral density-based and measure-preserving ABC for partially observed diffusion processes. An illustration on Hamiltonian SDEs.
\newblock{\em Stat. Comput.,} 30: 627--648, 2020.


\bibitem{Cardy}
J. L. Cardy.
\newblock Critical percolation in finite geometries.
\newblock{\em J. Phys. A, Math. Gen.,} 25 (4): 201--206, 1992.


\bibitem{Cardy2}
J. L. Cardy.
\newblock SLE for theoretical physicists.
\newblock{\em Ann. Phys.,} 318 (1): 81--118, 2005.


\bibitem{Castro/Lukovic/Pompanin/Andrade/Herrmann}
C. P. de Castro, M. Luković, G. Pompanin, R. F. S. Andrade, H. J. Herrmann.
\newblock Schramm--Loewner evolution and perimeter of percolation clusters of correlated random landscapes.
\newblock{\em Sci. Rep.,} 8, 5286, 2018.


\bibitem{Chelkak/Duminil-Copin/Hongler/Kemppainen/Smirnov}
D. Chelkak, H. Duminil-Copin, C. Hongler, A. Kemppainen, S. Smirnov.
\newblock Convergence des interfaces d'Ising vers les courbes SLE introduites par Schramm.
\newblock{\em C. R. Acad. Sci. Paris, Ser. I-Math.,} 352: 157--161, 2014.


\bibitem{Chen1}
J. Chen, V. Margarint.
\newblock Perturbations of multiple Schramm–Loewner evolution with two non-colliding Dyson Brownian motions.
\newblock{\em Sto. Process. Appl.,} 151: 553--569 (2022).


\bibitem{Chen2}
J. Chen.
\newblock On the localization regime of high-dimensional directed polymers in time-correlated random field.
\newblock{\em arXiv preprint,} arXiv:2412.14712, 2024.


\bibitem{Chen/Laulin}
J. Chen, L. Laulin.
\newblock Analysis of the smoothly amnesia-reinforced multidimensional elephant random walk.
\newblock{\em J. Stat. Phys.,} 190, 158, 2023.


\bibitem{Clark/Cameron}
J. M. C. Clark, R. J. Cameron.
\newblock The maximum rate of convergence of discrete approximations for stochastic differential equations.
\newblock{\em Stoch. Diff. Syst. Filter. Control, Lect. Note. Control Inform. Sci.,} 25, 1980.


\bibitem{Cohen/Elliott}
S. N. Cohen, R. J. Elliott.
\newblock Stochastic Calculus and Applications.
\newblock{\em Springer Probab., Appl., 2nd ed.,} 2015.


\bibitem{Credidio/Moreira/Herrmann/Andrade}
H. F. Credidio, A. A. Moreira, H. J. Herrmann, J. S. Andrade.
\newblock Stochastic Loewner evolution relates anomalous diffusion and anisotropic percolation.
\newblock{\em Phys. Rev. E,} 93, 042124, 2016.


\bibitem{Darses}
S. Darses.
\newblock Time reversal for drifted fractional Brownian motion with Hurst index $H > 1/2$.
\newblock{\em Electron. J. Probab.,} 12 (43): 1181--1211, 2007.


\bibitem{de Branges}
L. de Branges.
\newblock A proof of the Bieberbach conjecture.
\newblock{\em Acta Math.,} 154 (1-2): 137--152, 1985.


\bibitem{Dennis/Schnabel}
J. E. Dennis, R. B. Schnabel.
\newblock Numerical Methods for Unconstrained Optimization and Nonlinear Equations.
\newblock{\em Classic. Appl. Math., Illustrated ed.,} 16, 1987.


\bibitem{Duplantier/Kwon}
B. Duplantier, K.-H. Kwon.
\newblock Conformal invariance and intersections of random walks.
\newblock{\em Phys. Rev. Lett.,} 61 (22): 2514--2517, 1988.


\bibitem{Duplantier/Sheffield}
B. Duplantier, S. Sheffield.
\newblock Schramm--Loewner evolution and Liouville quantum gravity.
\newblock{\em Phys. Rev. Lett.,} 107, 131305, 2011.


\bibitem{Evans}
L. C. Evans.
\newblock An Introduction to Stochastic Differential Equations.
\newblock{\em American Math. Soc., 1st ed.,} 2014.


\bibitem{Folland}
G. B. Folland.
\newblock Real Analysis: Modern Techniques and Their Applications.
\newblock{\em Pure Appl. Math., 2nd ed.,} 1984.


\bibitem{Foster} 
J. Foster.
\newblock Numerical simulation of Schramm--Loewner Evolution.
\newblock{\em GitHub Reposit.,} https://github.com/james-m-foster/sle-simulation 2020.


\bibitem{Foster/Reis/Strange} 
J. Foster, G. dos Reis, C. Strange.
\newblock High order splitting methods for SDEs satisfying a commutativity condition.
\newblock{\em SIAM J. Numeri. Analy.,} 62, 10.1137/23M155147X, 2024.


\bibitem{Foster/Lyons/Margarint} 
J. Foster, T. Lyons, V. Margarint.
\newblock An asymptotic radius of convergence for the Loewner equation and simulation of SLE$_{\kappa}$ traces via splitting.
\newblock{\em J. Stat. Phys.,} 189, 18, 2022.


\bibitem{Friz/Victoir} 
P. K. Friz, N. B. Victoir.
\newblock Multidimensional Stochastic Processes as Rough Paths.
\newblock{\em Cambridge Studi. Adv. Math.,} 120, 2010.


\bibitem{Ghasemi Nezhadhaghighi/Rajabpour/Rouhani} 
M. Ghasemi Nezhadhaghighi, M. A. Rajabpour, S. Rouhani.
\newblock First-passage-time processes and subordinated Schramm--Loewner evolution.
\newblock{\em Phys. Rev. E}, 84, 011134, 2011.


\bibitem{Guan/Winkel} 
Q.-Y. Guan, M. Winkel.
\newblock SLE and $\alpha$-SLE driven by Lévy processes.
\newblock{\em Ann. Probab.,} 36 (4): 1221--1266, 2008.


\bibitem{Herrmann/Stanley} 
H. J. Herrmann, H. E. Stanley.
\newblock The fractal dimension of the minimum path in two- and three-dimensional percolation.
\newblock{\em J. Phys. A, Math. Gen.,} 21: 829--833, 1988.


\bibitem{Johansson/Sola} 
F. Johansson, A. Sola.
\newblock Rescaled Lévy--Loewner hulls and random growth.
\newblock{\em Bull. Sci. Math.,} 133 (3): 238--256, 2009.


\bibitem{Kager/Nienhuis/Kadanoff} 
W. Kager, B. Nienhuis, L. Kadanoff.
\newblock Exact solutions for Loewner evolutions.
\newblock{\em J. Stat. Phys.,} 115: 805--822, 2004.


\bibitem{Kemppainen}
A. Kemppainen. 
\newblock Schramm--Loewner Evolution.
\newblock{\em Springer Brief. Math. Phys.,} 24, 2017.


\bibitem{Kennedy0}
T. Kennedy. 
\newblock A fast algorithm for simulating the chordal Schramm--Loewner evolution.
\newblock{\em J. Stat. Phys.,} 128: 1125--1137, 2007.


\bibitem{Kennedy00}
T. Kennedy. 
\newblock Computing the Loewner driving process of random curves in the half plane.
\newblock{\em J. Stat. Phys.,} 131: 803--819, 2008.


\bibitem{Kennedy}
T. Kennedy. 
\newblock Numerical computations for the Schramm--Loewner evolution.
\newblock{\em J. Stat. Phys.,} 137: 839--856, 2009.


\bibitem{Kious/Mailler/Schapira}
D. Kious, C. Mailler, B. Schapira. 
\newblock Finding geodesics on graphs using reinforcement learning.
\newblock{\em Ann. Appl. Probab.,} 32 (5): 3889--3929, 2022.


\bibitem{Kious/Mailler/Schapira2}
D. Kious, C. Mailler, B. Schapira. 
\newblock The trace-reinforced ants process does not find shortest paths.
\newblock{\em J. Éco. Polytech. Math.,} 9: 505--536, 2022.


\bibitem{Kuhnau}
R. Kühnau. 
\newblock Numerische realisierung konformer abbildungen durch ``Interpolation''.
\newblock{\em Z. Angew. Math. Mech.,} 63: 631--637, 1983.


\bibitem{Langlands/Pouliot/Saint-Aubin}
R. Langlands, P. Pouliot, Y. Saint-Aubin.
\newblock Conformal invariance in two-dimensional percolation.
\newblock{\em Bull. American Math. Soc.,} 30: 1--61, 1994.


\bibitem{Lawler}
G. Lawler.
\newblock Conformally Invariant Processes in the Plane.
\newblock{\em American Math. Soc, Math. Sur. Mon.,} 114, 2005.


\bibitem{Lawler/Limic}
G. Lawler, V. Limic.
\newblock Random Walk: A Modern Introduction.
\newblock{\em Cambridge Studi. Adv. Math.,} 123, 2012.


\bibitem{Lawler/Schramm/Werner}
G. Lawler, O. Schramm, W. Werner.
\newblock Conformal invariance of planar loop-erased random walks and uniform spanning trees.
\newblock{\em Ann. Probab.,} 32 (1B): 939--995, 2004.


\bibitem{Lawler/Sheffield}
G. Lawler, S. Sheffield.
\newblock A natural parametrization for the Schramm--Loewner evolution.
\newblock{\em Ann. Probab.,} 39 (5): 1896--1937, 2011.


\bibitem{Lawler/Viklund}
G. Lawler, F. Viklund.
\newblock Convergence of loop-erased random walk in the natural parameterization.
\newblock{\em Duke Math. J.,} 170 (10): 2289--2370, 2021.


\bibitem{Limic/Tarres}
V. Limic, P. Tarrès.
\newblock Attracting edge and strongly edge reinforced walks.
\newblock{\em Ann. Probab.,} 35 (5): 1783--1806, 2007.


\bibitem{Lind/Rohde}
J. Lind, S. Rohde.
\newblock Spacefilling curves and phases of the Loewner equation.
\newblock{\em Indiana Univ. Math. J.,} 61 (6): 2231--2249, 2012.


\bibitem{Lind/Utley}
J. Lind, J. Utley.
\newblock Phase transition for a family of complex-driven Loewner hulls.
\newblock{\em Involve,} 15 (3): 447--474, 2022.


\bibitem{Loewner}
C. Loewner.
\newblock Untersuchungen über schlichte konforme abbildungen des einheitskreises. I.
\newblock{\em Math. Ann.,} 89: 103--121, 1923.


\bibitem{LeGall}
J.-F. Le Gall.
\newblock Brownian Motion, Martingales, and Stochastic Calculus.
\newblock{\em Springer Grad. Text. Math.,} 274, 2016.


\bibitem{Liu/Wu}
M. Liu, H. Wu.
\newblock Loop-erased random walk branch of uniform spanning tree in topological polygons.
\newblock{\em Bernoulli,} 29 (2): 1555--1577, 2023.


\bibitem{Mandelbrot/van Ness}
B. B. Mandelbrot, J. W. van Ness.
\newblock Fractional Brownian motions, fractional noises and applications.
\newblock{\em SIAM Rev.,} 10 (4): 422--437, 1968.


\bibitem{Margarint/Shekhar/Yuan}
V. Margarint, A. Shekhar, Y. Yuan.
\newblock On Loewner chains driven by semimartingales and complex Bessel-type SDEs.
\newblock{\em Ann. Appl. Probab.,} to appear, 2021.


\bibitem{Marshall/Rohde}
D. E. Marshall, S. Rohde.
\newblock Convergence of a variant of the zipper algorithm for conformal mapping.
\newblock{\em SIAM J. Numer. Analy.,} 45: 2577--2609, 2007.


\bibitem{Mattingly/Stuart/Higham}
J. C. Mattingly, A. M. Stuart, D. J. Higham.
\newblock Ergodicity for SDEs and approximations: locally Lipschitz vector fields and degenerate noise.
\newblock{\em Stoch. Process. Appl.,} 101: 185--232, 2002.


\bibitem{Merkl/Rolles}
F. Merkl, S. W. W. Rolles.
\newblock Linearly edge-reinforced random walks.
\newblock{\em IMS Lect. Note.-Mon. Ser. Dyn. Stoch.,} 48: 66–77, 2006.


\bibitem{Miller/Sheffield}
J. Miller, S. Sheffield.
\newblock Quantum Loewner evolution.
\newblock{\em Duke Math. J.,} 165 (17): 3241--3378, 2016.


\bibitem{Najafi}
M. Najafi.
\newblock Left passage probability of Schramm-Loewner Evolution.
\newblock{\em Phys. Rev. E,} 87, 062105, 2013.


\bibitem{Najafi2}
M. Najafi.
\newblock Fokker--Planck equation of Schramm-Loewner evolution.
\newblock{\em Phys. Rev. E 92,} 022113, 2015.


\bibitem{Najafi/Tizdast/Cheraghalizadeh}
M. Najafi, S. Tizdast, J. Cheraghalizadeh.
\newblock Schramm--Loewner evolution in the random scatterer Henon-percolation landscapes.
\newblock{\em Acta Phys. Polonica B,} 50: 929--942, 2019.


\bibitem{Oikonomou/Rushkin/Gruzberg/Kadanoff}
P. Oikonomou, I. Rushkin, I. A. Gruzberg, L. P. Kadanoff.
\newblock Global properties of stochastic Loewner evolution driven by Lévy processes.
\newblock{\em J. Stat. Mech.: Theor. Exp.,} P01019, 2008.


\bibitem{Oliveira/Schrenk/Araujo/Herrmann/Andrade}
E. A. Oliveira, K. J. Schrenk, N. A. M. Araújo, H. J. Herrmann, J. S. Andrade.
\newblock Optimal-path cracks in correlated and uncorrelated lattices.
\newblock{\em Phys. Rev. E,} 83, 046113, 2011.


\bibitem{Peltola/Schreuder}
E. Peltola, A. Schreuder.
\newblock Loewner traces driven by Lévy processes.
\newblock{\em arXiv preprint,} arXiv:2407.06144, 2024.


\bibitem{Pemantle}
R. Pemantle.
\newblock A survey of random processes with reinforcement.
\newblock{\em Probab. Survey.,} 4: 1--79, 2007.


\bibitem{Petersen}
W. P. Petersen.
\newblock A general implicit splitting for stabilizing numerical simulations of Ito stochastic differential equations. 
\newblock{\em SIAM J. Numer. Analy.,} 35: 1439--1451, 1998.


\bibitem{Pommerenke}
C. Pommerenke.
\newblock Boundary Behaviour of Conformal Maps.
\newblock{\em Grundlehren Math. Wiss.,} 299, 1992.


\bibitem{Reed/Lee/Truong}
I. S. Reed, P. C. Lee, T. K. Truong.
\newblock Spectral representation of fractional Brownian motion in $n$ dimensions and its properties.
\newblock{\em IEEE Trans. Inform. Theor.,} 41 (5): 1439--1451, 1995.


\bibitem{Rohde/Schramm}
S. Rohde, O. Schramm.
\newblock Basic properties of SLE.
\newblock{\em Ann. Math.,} 161 (2): 883--924, 2005.


\bibitem{Rohde/Zhan}
S. Rohde, D. Zhan.
\newblock Backward SLE and the symmetry of the welding.
\newblock{\em Probab. Theor. Relat. Field.,} 164: 815--863, 2016.


\bibitem{Rosales-Ortiz}
A. Rosales-Ortiz.
\newblock Noise reinforced Lévy processes: Lévy--Itô decomposition and applications.
\newblock{\em Electron. J. Probab.,} 28: 1-58, 2023.


\bibitem{Rushkin/Oikonomou/Kadanoff/Gruzberg}
I. Rushkin, P. Oikonomou, L. P. Kadanoff, I. A. Gruzberg.
\newblock Stochastic Loewner evolution driven by Lévy processes.
\newblock{\em J. Stat. Mech.: Theor. Exp.,} P01001, 2006.


\bibitem{Sabot/Tarres}
C. Sabot, P. Tarrès. 
\newblock Edge-reinforced random walk, vertex-reinforced jump process and the supersymmetric hyperbolic sigma model.
\newblock{\em J. European Math. Soc.,} 17 (9): 2353--2378, 2015.


\bibitem{Schramm}
O. Schramm. 
\newblock Scaling limits of loop-erased random walks and uniform spanning tree.
\newblock{\em Israel J. Math.,} 118: 221--288, 2000.


\bibitem{Schramm/Sheffield}
O. Schramm, S. Sheffield. 
\newblock Contour lines of the two-dimensional discrete Gaussian free field.
\newblock{\em Acta Math.,} 202 (1): 21--137, 2009.


\bibitem{Shekhar/Tran/Wang}
A. Shekhar, H. Tran, Y. Wang. 
\newblock Remarks on Loewner chains driven by finite variation functions.
\newblock{\em Ann. Acad. Scient. Fennicæ,} 44 (1): 311--327, 2019. 


\bibitem{Schmittbuhl/Vilotte/Roux}
J. Schmittbuhl, J.-P. Vilotte, S. Roux. 
\newblock Percolation through self-affine surfaces.
\newblock{\em J. Phys. A, Math. Gen.,} 26: 6115--6133, 1993.


\bibitem{Shevchenko}
G. Shevchenko. 
\newblock Fractional Brownian motion in a nutshell.
\newblock{\em Int. J. Mod. Phys.,} 36, 1560002, 2015.


\bibitem{Tanaka/Kayama/Kato/Ito}
M. Tanaka, A. Kayama, R. Kato, Y. Ito.
\newblock Estimation of the fractal dimension of fracture surface patterns by box-counting method.
\newblock{\em Fractals,} 7 (3): 335--340, 1999.


\bibitem{Tizdast/Ebadi/Cheraghalizadeh/Najafi/Andrade/Herrmann}
S. Tizdast, Z. Ebadi, J. Cheraghalizadeh, M. N. Najafi, J. S. Andrade, H. J. Herrmann. 
\newblock Self-similar but not conformally invariant traces obtained by modified Loewner forces.
\newblock{\em Phys. Rev. E,} 105, 024103, 2022. 


\bibitem{Tran} 
H. Tran.
\newblock Convergence of an algorithm simulating Loewner curves.
\newblock{\em Ann. Acad. Scient. Fennicæ,} 40: 601--616, 2015.


\bibitem{Viklund/Lawler}
F. J. Viklund, G. Lawler.
\newblock Optimal Hölder exponent for The SLE path.
\newblock{\em Duke Math. J.,} 159: 351--383, 2011.


\bibitem{Viklund/Rohde/Wong}
F. J. Viklund, S. Rohde, C. Wong.
\newblock On the continuity of SLE$_\kappa$ in $\kappa$.
\newblock{\em Probab. Theor. Relat. Field.,} 159: 413--433, 2014. 


\bibitem{Weinrib/Halperin}
A. Weinrib, B. I. Halperin.
\newblock Critical phenomena in systems with long-range-correlated quenched disorder.
\newblock{\em Phys. Rev. B,} 27 (1): 413--427, 1983.


\bibitem{Weiss}
J. Weiss.
\newblock Self-affinity of fracture surfaces and implications on a possible size effect on fracture energy.
\newblock{\em Int. J. Fract.,} 109 (4): 365--381, 2001.


\bibitem{Werner}
W. Werner.
\newblock Random planar curves and Schramm--Loewner evolutions.
\newblock{\em Springer Lect. Probab. Theor. Stat.,} 1840: 107--195, 2004.




\end{thebibliography}
\begin{spacing}{1}

\end{spacing}

\end{document}